\documentclass[suppldata]{interact}
\usepackage{epstopdf}
\usepackage{subfigure}

\usepackage{natbib}
\bibpunct[, ]{(}{)}{;}{a}{}{,}
\usepackage{hyperref}
\usepackage{bbm, comment}
\usepackage{color}

\theoremstyle{plain}
\newtheorem{theorem}{Theorem}[section]
\newtheorem{lemma}[theorem]{Lemma}

\newtheorem{proposition}[theorem]{Proposition}

\theoremstyle{definition}
\newtheorem{definition}[theorem]{Definition}

\theoremstyle{remark}
\newtheorem{remark}{Remark}

\newcommand{\tensor}[1]{\mathsf{#1}}

 \DeclareMathOperator{\RE}{Re}
 \DeclareMathOperator{\IM}{Im}
 
 \DeclareMathAlphabet{\mathpzc}{OT1}{pzc}{m}{it}
 \DeclareMathAlphabet{\mathsfsl}{OT1}{cmss}{m}{sl}

  \newcommand{ \mi}{\mathrm{i}}
  
  \newcommand{\FH}{\mathfrak{H}}

\newcommand{\dif}{\mathrm{d}}

\newcommand{\R}{\mathbb{R}}
 \newcommand{\Rnum}{\mathbb{R}}

\newcommand{\abs}[1]{\left\vert#1\right\vert}
\newcommand{\set}[1]{\left\{#1\right\}}

\newcommand{\norm}[1]{\left\Vert#1\right\Vert}
 \newcommand{\innp}[1]{\langle {#1}\rangle}

\newcommand{\E}{\mathbb{E}}
\newcommand{\1}{\mathbf{1}}
\newcommand\sgn{\mathrm{sgn}}

\allowdisplaybreaks[4]
\begin{document}

\title{
Parameter Estimation for the Complex Fractional Ornstein-Uhlenbeck Processes with Hurst parameter $H \in (0, \frac12)$}
\author{
\name{Fares ALAZEMI\textsuperscript{a,*}\thanks{*Corresponding author:  Fares ALAZEMI, Email: fares.alazemi@ku.edu.kw. } , Abdulaziz ALSENAFI\textsuperscript{a}, Yong CHEN\textsuperscript{b}, and  Hongjuan ZHOU\textsuperscript{c}}
\affil{\textsuperscript{a}Department of Mathematics, Faculty of Science, Kuwait University, Kuwait. \\
\textsuperscript{b}School of Mathematics and Statistics, Jiangxi Normal University, Nanchang, 330022, Jiangxi, China.}
\textsuperscript{c}{School of Mathematical and Statistical Sciences, Arizona State University,\,Arizona,  {85287},  USA}
}

\maketitle
\begin{abstract}
We study the strong consistency and asymptotic normality of a least
squares estimator of the drift coefficient in complex-valued Ornstein-Uhlenbeck processes driven by fractional Brownian motion, extending  the results  of \cite{chw} to the case of Hurst parameter $H\in (\frac14,\frac12)$ and the results of \cite{hnz 19} to a two-dimensional case. When $H\in (0,\frac14]$,  it is found that the integrand of the estimator is not in the domain of the standard divergence operator. To facilitate the proofs, we develop a new inner product formula for functions of bounded variation in the reproducing kernel Hilbert space of fractional Brownian motion with Hurst parameter $H\in (0,\frac12)$. This formula is also applied to obtain the second moments of the so-called $\alpha$-order fractional Brownian motion and the $\alpha$-fractional bridges with the Hurst parameter $H\in (0,\frac12)$.
\end{abstract}
\begin{amscode} 60G15; 60G22; 62M09 \end{amscode}

\section{Introduction and main results}

The statistical inference problems for one-dimensional stochastic differential equations driven by fractional Brownian motion have been intensively studied in the literature, but the statistical estimations for the multi-dimensional fractional stochastic equation have not been completely studied (see \cite{hnz 19a} and the references therein). This paper seeks to make a meaningful contribution within the context of this landscape, by studying the least squares estimator for the complex fractional Ornstein-Uhlenbeck process where the Hurst parameter is less than $\frac12$. 

The complex Ornstein-Uhlenbeck process is known as the solution to the stochastic differential equation 
\begin{equation}\label{cp}
  \dif Z_t=-\gamma Z_t\dif t+ \sqrt{a}\dif \zeta_t\,,\quad t\ge 0\,,
\end{equation}%
where $Z_t=X_1(t)+\mi X_2(t)$ is a complex-valued process, $\gamma=\lambda-\mi \omega,\,\lambda>0$, $a>0$ and $\zeta_t
$ is a complex Brownian motion. This process has been applied to model the Chandler wobble, or variation of latitude concerning the rotation of the earth (see \cite{arta3}, \cite{arat}). 

In \cite{chw},  the statistical estimator of $\gamma$ is considered when the complex Brownian motion $\zeta$ in  \eqref{cp} is replaced by a complex fractional Brownian motion
$$\zeta_t=\frac{B^1_t+i B^2_t}{\sqrt 2},$$ where
$(B^1_t, B^2_t)$ is a two-dimensional fractional Brownian motion (fBm) with Hurst parameter $H$.
Namely, setting up $a=1$ and $Z_0=0$ without loss of generality, the equation \eqref{cp} is expressed intuitively as
\begin{equation*}
   \dot{Z}_t+\gamma Z_t=\dot{\zeta}_t\,, \quad 0\le t\le T.
\end{equation*}
 Minimizing $\int_0^T\abs{ \dot{Z}_t+\gamma Z_t}^2\dif t$ yields a least squares estimator of $\gamma$ as follows:
\begin{equation}\label{hat gamma}
   \hat{\gamma}_T=-\frac{\int_0^T \bar{Z}_t\dif Z_t}{\int_0^T \abs{Z_t}^2\dif t}=\gamma-\frac{\int_0^T \bar{Z}_t\dif \zeta_t}{\int_0^T \abs{Z_t}^2\dif t}.
\end{equation} 
When $H\in [\frac12,\frac34)$, the strong consistency and the asymptotic normality of the estimator $\hat{\gamma}_T$ are shown in \cite{chw}. The work has been extended to the complex fractional Vasicek model (see \cite{sty}). 

We would like to point out that 
$\hat{\gamma}_T$ is not an ideal estimator as both stochastic integrals in \eqref{hat gamma} are related to the unknown parameter $\gamma$. However, it is still meaningful to study the asymptotic properties of the ratio process 
\begin{align}\label{ratio process}
    \gamma-\hat{\gamma}_T=\frac{\int_0^T \bar{Z}_t\dif \zeta_t}{\int_0^T \abs{Z_t}^2\dif t}
\end{align}
that could be useful in deriving other estimators (\cite{hnz 19a}), for example, the ergodic estimator or moment estimator. 

The question naturally arises whether the strong consistency and the asymptotic normality of the estimator $\hat{\gamma}_T$ still hold when $H\in (0,\frac12)$. An affirmative answer is shown for the real-valued fractional Ornstein-Uhlenbeck process in \cite{hnz 19}. Hence, it is conjectured that the result is also valid for the complex-valued fractional Ornstein-Uhlenbeck process. However, the answer is not as expected , i.e., it turns out that part of the results in the real-valued fractional Ornstein-Uhlenbeck process cannot be extended into the complex case when $0<H<\frac{1}{2}$. The results for the complex case are stated in the following Theorem \ref{com-ou parameter estimate}.



\begin{theorem}\label{com-ou parameter estimate}
   Let $H\in (\frac14, \frac34)$.
   \begin{enumerate}
   \item[(i)]
$\hat{\gamma}_T$  converges to $\gamma$ almost surely
as $T\rightarrow \infty$.
\item[(ii)] $\sqrt{T}\left( \hat{\gamma}_T-\gamma\right)  $ is asymptotically bivariate normal. Namely,
\begin{equation}\label{normality.gammat}
   \sqrt{T}[\hat{\gamma}_T-\gamma]\stackrel{ {law}}{\to}  \mathcal{N} \left(0,\frac{1}{ d^2 } \tensor{C}\right)\quad\hbox{
   as $T\rightarrow \infty$}\,,
\end{equation}
where
$\tensor{C}=\begin{bmatrix} \sigma^2+c & b\\ b & \sigma^2-c  \end{bmatrix}$ with
 \begin{align}
	    \sigma^2&= \frac{1}{2
     \lambda}\big( {\gamma^{2-4H}}+ {\bar{\gamma}^{2-4H}}\big) \left( 1+\frac{\Gamma(3-4H)\Gamma(4H-1)}{\Gamma(2H)\Gamma(2-2H)}\right),\label{sigmah0002} \\
     c +\mi b &=\frac{4H-2}{
     \bar{\gamma}^{4H-1}}\left( 1+\frac{\Gamma(3-4H)\Gamma(4H-1)}{\Gamma(2H)\Gamma(2-2H)}\right), \label{sigmajiaocha0002}\\
     d&= \frac{1}{2\lambda}\big( {\gamma^{1-2H}}+ {\bar{\gamma}^{1-2H}}\big)\,. 
\label{d d}
	\end{align} 
\end{enumerate}
\end{theorem}
\begin{remark} 
   When $H\in [\frac12, \frac34)$, the results 
   are consistent with those in \cite{chw}. When $H\in (0,\frac14]$,  the stochastic integral in the estimator \eqref{hat gamma} is not 
   well defined (see Proposition~\ref{qq1a}).
\end{remark}
The proof for the strong consistency of the estimator $\hat\gamma_T$ is based on ergodicity and the Garsia-Rodemich-Rumsey inequality. The complex fourth moment theorem (see Theorem~\ref{com FMT}) can be applied to show the asymptotically normality of $\hat\gamma_T$. Namely, we 
introduce two complex $(1,1)$ Wiener-It\^o integrals (see Definition~\ref{def.complex.ito}):
\begin{equation}
    X_T=I_{1,1}(\psi_T(t,s)),\quad  F_T=\frac{1}{\sqrt{T}}I_{1,1}(\psi_T(t,s)),
\label{e.def-X}
\end{equation} where the kernel $\psi_T$ is given in \eqref{phist defn}. By taking into account that $\bar Z_t = \int_0^t \psi_T (t,s) d\zeta_s$, we rewrite \eqref{ratio process} as 
\begin{equation} 
    \hat{\gamma}_T-\gamma =-\frac{\frac{1}{ {T}}X_T}{\frac{1}{T}\int_0^T \abs{Z_t}^2\dif t}, \label{ratio.pro.re}\\
\end{equation}
and
\begin{align}
     \sqrt{T}(\hat{\gamma}_T-\gamma)&=-\frac{F_T}{\frac{1}{T}\int_0^T \abs{Z_t}^2\dif t}. \label{ratio.pro.norm}
\end{align}
Denote by $\mathfrak{H}$ the associated reproducing kernel Hilbert space of fBm $B^H$. With the abuse of notation, we still use $\mathfrak{H}$ to denote its complexification. Following the proof idea in \cite{hnz 19}, we can use Fourier transform to bound the norm of $\psi_T$ in $\mathfrak{H}^{\otimes 2}$, and the contraction between the functions $\psi_T$ and $h_T$, where $\psi_T,\,h_T$ are given in \eqref{phist defn}. The results are summarized in the following Proposition \ref{qq1a}. This proposition highlights an important result that can be helpful to prove the limiting theorem given in \eqref{normality.gammat} for fulfilling the key conditions required by the complex fourth moment theorem (see Theorem~\ref{com FMT}). 

 \begin{proposition}\label{qq1a}
         Denote the functions of two variables
	\begin{equation}\label{phist defn}
  \psi_T(t,s)=e^{-\bar{\gamma}(t-s)}\mathbf{1}_{\{0\le s\le t\le T\}} , \qquad       h_T(t,s)=e^{- {\gamma}(s-t)}\mathbf{1}_{\{0\le t\le s\le T\}}.
	\end{equation}
  If $H\in(0, \frac14)$ then $\psi_T, h_{T}$ do not belong to the tensor space $\FH^{\otimes 2}$. If $H\in(\frac14,\frac12)$, then $\psi_T, h_{T}$ 
 belong to the tensor space $\FH^{\otimes 2}$ and there is a positive constant $C_{H,\theta}$ independent of $T$ such that when $T$ is large enough, the inequalities 
\begin{align}
     \Big|	\|\psi_T\|^2_{{\FH}^{\otimes2}}-(H\Gamma(2H))^2 M_H^2 T\Big|&\le  C_{H,\theta}\label{diyigedingli-complex},\\
  \Big|  \innp{\psi_T, h_{T}}_{{\FH}^{\otimes2}} - (H\Gamma(2H))^2  N_H T  \Big| &\le  C_{H,\theta}, \label{diyigedingli-complex-002}
\end{align}
 hold, where
 \begin{align}
	    M_H^2&= \frac{1}{2
     \lambda}\big( {\gamma^{2-4H}}+ {\bar{\gamma}^{2-4H}}\big) \left( 1+\frac{\Gamma(3-4H)\Gamma(4H-1)}{\Gamma(2H)\Gamma(2-2H)}\right),\label{sigmah0002} \\
     N_H&=\frac{4H-2}{
     \bar{\gamma}^{4H-1}}\left(1+\frac{\Gamma(3-4H)\Gamma(4H-1)}{\Gamma(2H)\Gamma(2-2H)}\right). \label{sigmajiaocha0002}
	\end{align}
 \end{proposition}
The computations of the inner products for the functions in Proposition \ref{qq1a} 
are based on a simplified inner product formula of fBm with Hurst parameter $H\in (0,\frac12)$, see \eqref{innp fg3-0}-\eqref{innp fg3-00}. 
In \cite{hnz 19}, the following formula provides a computation for the inner product of two functions with support on $[0, T]$ in the Hilbert space $\mathfrak{H}$:
\begin{align}\label{zh-n-hu}
    \langle f,\,g \rangle_{\FH}&=-\int_{[0,T]^2}  f(t) g'(s)  \frac{\partial R(s,t)}{\partial t}  \dif t \dif s
\end{align} 
where $$R(s,t)=\frac{1}{2}[s^{2H}+t^{2H}-\abs{s-t}^{2H}]$$ is the covariance function of the fBm, 
and 
 \begin{align*}
    \frac{\partial R(s,t)}{\partial t}=H\big[t^{2H-1} -\abs{t-s}^{2H-1}\sgn(t-s)\big].
\end{align*} 
The derivative $g'$ in \eqref{zh-n-hu} should be understood as the distribution derivative of the function $g$, or more precisely, $g'(s) \dif s$ is the measure $\nu_g(\dif s)$ (see Proposition~\ref{st1-thm-original}). In this paper, we will show that the inner product formula \eqref{zh-n-hu} can be simplified using the fact that the term $t^{2H-1}$ in the partial derivative $\frac{\partial R(s,t)}{\partial t}$
does not contribute to the integration value given by \eqref{zh-n-hu}. This simplified formula is summarized in the following Proposition \ref{st1-thm-original}.
 
\begin{proposition}\label{st1-thm-original}
 Denote $\mathcal{V}_{[0,T]}$ as the set of bounded variation functions on $[0,T]$. Let $H \in (0, \frac{1}{2})$. For any two functions in the set $\mathcal{V}_{[0,T]}$, their inner product in the Hilbert space $\mathfrak{H}$ can be expressed as
 \begin{align} 
\langle f,\,g \rangle_{\FH}
&=H \int_{[0,T]^2}  f(t)  \abs{t-s}^{2H-1}\sgn(t-s) \dif t \nu_{g}(\dif s),\quad  \forall f,\, g\in \mathcal{V}_{[0,T]}, \label{innp fg3-0}
\end{align}   where $\nu_{g}(\dif s):=\dif \nu_{g}(s)$, and $\nu_g$ is the restriction on $\left({[0,T]},\mathcal{B}({[0,T]})\right)$ of the signed \textnormal{Lebesgue-Stieljes} measure $\mu_{g^0}$ on $\left(\Rnum,\mathcal{B}(\Rnum)\right)$ of
\begin{equation*}
g^0(x)=\left\{
    \begin{array}{ll}
g(x), & \quad \text{if}~x\in [0,T],\\
0, &\quad \text{otherwise}.
 \end{array}
\right.
\end{equation*}
If $  g'(\cdot) $ is interpreted as the distributional derivative of $g(\cdot)$, then the formula \eqref{innp fg3-0} admits the following representation:
 \begin{align} 
\langle f,\,g \rangle_{\FH}
&=H \int_{[0,T]^2}  f(t) g'(s) \abs{t-s}^{2H-1}\sgn(t-s) \dif t  \dif s,\quad  \forall f,\, g\in \mathcal{V}_{[0,T]}. \label{innp fg3-00}
\end{align}  
\end{proposition}

The formula  \eqref{innp fg3-00} is novel to our best knowledge. To demonstrate the usefulness of the above inner product formulae, we will show two applications for computing the second moments for $\alpha$-order fBm and the $\alpha$-fractional bridges when $H\in (0,\frac12)$, which are not yet solved till now (see \cite{el2023} and \cite{Es-Sebaiy13}).

\begin{theorem}\label{fbm-bridge exist}
Assume $H\in (0,1)$ and $\alpha\in (0,H)$. Define the stochastic process $\xi_t$ as
\begin{align}\label{xit dingyi}
\xi_t=\int_0^t (T-u)^{-\alpha} \dif B^H_u,\qquad \quad 0\le t< T. 
\end{align}
We have that
$\xi_T:=\lim_{t\uparrow T}\xi_t$ exists in $L^2$ and almost surely, and that 
\begin{align}\label{xiT2qiw}
\E[\xi_T^2]=\frac{H}{H-\alpha}\frac{\Gamma(1-\alpha)\Gamma(2H)}{\Gamma(2H-\alpha)} T^{2(H-\alpha)}.
\end{align}
Moreover, the Gaussian process $(\xi_t)_{t\in[0,T]}$  admits a modification on $[0,T]$ with $(H-\alpha-\epsilon)$-H\"{o}lder continuous paths. 
\end{theorem}

\begin{remark}
    In \cite{el2023}, the process $\xi$ is named as $\alpha'$-order fractional Brownian motion if $\alpha':=-\alpha\in (-1, \infty)$. No singularity appears in the equation \eqref{xit dingyi} when $\alpha \le 0$. If $\alpha \in (0, 1)$ and $H\in (\frac12,1)$, the second moment of $\xi_t$ has been studied in \cite{el2023, Es-Sebaiy13}. They claim as $\alpha \to 0$, the process $\xi$ retrieves the standard fBm when $H> \frac12$ and $\alpha\in (0,H)$. In this paper, we show that it is also true for  $H\in(0, \frac12)$. 
\end{remark}

As another application of the Proposition~\ref{st1-thm-original}, we will consider the $\alpha$-fractional Brownian bridge $Y_t$ (see \cite{Es-Sebaiy13}),
\begin{align*}
    Y_t=(T-t)^{\alpha}\int_0^t (T-u)^{-\alpha} \dif B^H_u,\quad \quad 0\le t\le  T,
\end{align*}
which is the solution to the following equation:
\begin{align*}
   \dif Y_t=-\alpha \frac{Y_t}{T-t}\dif t+ \dif B^H_t\,,\quad 0\le t<T\,;  \quad \alpha>0, \ Y_0=0.
\end{align*}
The second moment of a certain scaling limit of $Y_t$ has been studied in \cite{Es-Sebaiy13} for the case of $H\in (\frac12,1)$. By applying \eqref{innp fg3-00}, we obtain the result for the case $H \in (0, \frac12)$ in Theorem \ref{fbm bridge 2 exist}.

\begin{theorem}\label{fbm bridge 2 exist}
Assume $H\in (0,1)$ and $\alpha\in (H,1)$. Define 
\begin{align}\label{eta dingyi}
\tilde{Y}_t:=\frac{Y_t}{(T-t)^H}=(T-t)^{\alpha-H} \int_0^t (T-u)^{-\alpha} \dif B^H_u,\qquad \quad 0\le t< T. 
\end{align} We have that $\tilde{Y}_T:=\lim_{t\uparrow T}\tilde{Y}_t$ exists in $L^2$
and that
\begin{align}
\E[\tilde{Y}_T^2]=\frac{\alpha H}{\alpha-H}B(2H,\,1+\alpha-2H),\label{zetaT2qiw}
\end{align}
and
\begin{align}
 \E[B_s^H \tilde{Y}_T]=0,\quad \forall s\in [0,T]. \label{zetaT2qiw-2}
\end{align}
\end{theorem}

The paper is organized as follows. In Section~\ref{sec-002}, we introduce some elements about the complex isonormal Gaussian process and prove our new inner product formula \eqref{innp fg3-00} when $H\in (0,\frac12)$ given in Proposition ~\ref{st1-thm-original}. In Section~\ref{sec.3}, we will apply this inner product formula to prove Proposition~\ref{qq1a}. In Section~\ref{sec.4}, we will prove our main result Theorem~\ref{com-ou parameter estimate}, i.e., the strong consistency and the asymptotic normality of the LSE $\hat{\gamma}_T$. In Section~\ref{sec.5}, we will prove Theorem~\ref{fbm-bridge exist} and Theorem~\ref{fbm bridge 2 exist}, the two applications of our new inner product formula \eqref{innp fg3-00} for $H\in (0,\frac12)$. Several technical inequalities and asymptotic approximations of integrals are provided in the Appendix.

The symbol $C$ throughout the paper stands for a generic constant, whose value can change from one line to another. The notation $g(u)=O(1)$ means that there exist constants $M$  and $a$ such that the real-valued function $g$ satisfies $\abs{g(u)}\le M$ for all $u>a$.  For a positive function $\phi$, we say that a real-valued function $f$ satisfies $f=o(\phi)$ if $\frac{f(u)}{\phi(u)}\to 0$ as $u\to \infty$.

\section{Hilbert space associated with fBm and complex isonormal Gaussian process}\label{sec-002}

The fractional Brownian motion $(B^H_t)_{t\in \mathbb{R}}$ is defined on a complete probability space $(\Omega, \mathcal{F}, P)$. Denote by $\mathfrak{H}$ the associated reproducing kernel Hilbert space, 
which is defined
as the closure of the space of all real-valued step functions on $\mathbb{R}$ endowed with the inner product
\begin{align*}
    \langle \1_{[a,b]},\,\1_{[c,d]}\rangle_{\FH}=\E\big(( B^H_b-B^H_a) ( B^H_d-B^H_c) \big),
\end{align*}
under the convention that $\1_{[0, t]} = -\1_{[t, 0]}$. Denote the isonormal process on the same probability space $(\Omega, \mathcal{F}, P)$ by
$$B^H=\Big\{B^H(h)=\int_{\mathbb{R}}h(t)\dif B^H_t, \quad h \in \mathfrak{H}\Big\}.$$ It is indexed by the elements in the Hilbert space $\mathfrak{H}$, and satisfies the It\^{o}'s isometry:
\begin{align}\label{G extension defn}
\mathbb{E}(B^H(g)B^H(h)) = \langle g, h \rangle_{\mathfrak{H}}, \quad
\forall g, h \in \mathfrak{H}.
\end{align} 
If $H\in (\frac12,1)$ or the intersection of the supports of two elements $f,\,g \in \mathfrak{H}$ is of Lebesgue measure zero (see \cite{Mishura}), we have
 \begin{align} \label{innp fg3-zhicheng0}
\langle f,\,g \rangle_{\FH}=H(2H-1)\int_{\mathbb{R}^2}  f(t)g(s) \abs{t-s}^{2H-2} \dif t  \dif s. \end{align}

Next, denote $\mu_F$ the signed Lebesgue-Stieltjes measure of the bounded variation function $F$. Suppose that $[a,b]$ is a compact interval with positive length. Denote $\mathcal{V}_{[a,b]}$ as the set of bounded variation functions on $[a,b]$. For $g\in\mathcal{V}_{[a,b]}$, denote $\nu_g$ as the restriction on $\left({[a,b]},\mathcal{B}({[a,b]})\right)$ of the signed \textnormal{Lebesgue-Stieljes} measure $\mu_{g^0}$ on $\left(\Rnum,\mathcal{B}(\Rnum)\right)$   {of}
\begin{equation*}
g^0(x)=\left\{
    \begin{array}{ll}
g(x), & \quad \text{if}~x\in [a,b],\\
0, &\quad \text{otherwise}.
 \end{array}
\right.
\end{equation*}
The measure $\nu_g$ is similarly defined as in \cite{Jolis 07}. {It may have atoms, i.e.,
 $\nu_g(\{0\})=g(a),\, \nu_g(\{T\})=-g(b-)$
.} The following integration by parts formula pertaining to the measure $\nu_g$ is extracted from 
\cite{chenli2023}. Essentially, this formula is just a variation of integration by parts for bounded variation functions (see \cite{Foll 99, WZ 1977}).
\begin{lemma}\label{partial integral}
If $f: [a,b]\rightarrow\R$ is absolutely continuous on~$[a,b]$ and $g\in\mathcal{V}_{[a,b]}$,  then we have
\begin{align}
  -\int_{[a,b]}g(t) f'(t)\dif t=\int_{[a,b]}f(t) {\nu_g}(\dif t), \label{01}
\end{align}where $f'(t)$ is the  derivative of $f(t)$.
\end{lemma}

By taking $f$ as a constant function, Lemma~\ref{partial integral} implies
\begin{equation}\label{lamma21a}
   \int_{[a, b]} v_g(dt) = 0.
\end{equation}
We would like to point out the well-known integration result (see \cite[p.108]{Foll 99}) is a special case of the above lemma. Namely, for two absolutely continuous functions $f$ and $g$  on~$[a,b]$, 
\begin{align*}
    -\int_{[a,b]}g(t) f'(t)\dif t=\int_{[a,b]}f(t) g'(t)\dif t+f(a)g(a)-f(b)g(b).
\end{align*}
The right hand of the above identity is consistent with the integral in \eqref{01}, as
$$\int_{[a,b]}f(t) \dif \nu_g(t)=\int_{[a,b]}f(t)  \nu_g(\dif t),$$
and in this case $\nu_g(\dif t)=g'(t)\dif t+g(t)\big(\delta_a(t)-\delta_b(t)\big)\dif t$. 

Next, we will prove the novel inner product formula given in Proposition~\ref{st1-thm-original}.\\

{\bf Proof of Proposition~\ref{st1-thm-original}: } 
Theorem 2.3 of \cite{Jolis 07} implies that  $\forall f,\, g\in \mathcal{V}_{[0,T]}$, \begin{align}  \label{inner product 001}
\innp{f,g}_{\FH}=\int_{[0,T]^2} R(t,s) \nu_f( \dif t)   \nu_{g}( \dif s)=\int_0^T \Big( \int_0^T R(t,s) \nu_f( \dif t)  \Big) \nu_{g}( \dif s).
\end{align} 

Applying Lemma~\ref{partial integral} to the functions $R(\cdot, s)$ and $f(\cdot)$, we have
 \begin{align*} 
\langle f,\,g \rangle_{\FH}&=-\int_{[0,T]}  f(t)  \dif t \int_{[0,T]}\frac{\partial R(s,t)}{\partial t} \nu_{g}(\dif s)\notag \\
&=H \int_{[0,T]}  f(t)  \dif t \int_{[0,T]} \big[  \abs{t-s}^{2H-1}\sgn(t-s) -t^{2H} \big] \nu_{g}(\dif s) \\
&=H \int_{[0,T]}  f(t)  \dif t \int_{[0,T]} \abs{t-s}^{2H-1}\sgn(t-s) \nu_{g}(\dif s), 
\end{align*}where 
in the last line we have applied equation \eqref{lamma21a}.
 {\hfill\large{$\Box$}}

For example, the type of functions $g=h \1_{[a,b]}$ is used in this paper, where $0\le a<b\le T$ and $h$ is a differentiable function. In this case, the Lebesgue-Stieljes signed measure $\nu_{g}$ on  $([0,T ], \mathcal{B}([0,T ]))$ has an expression:
\begin{equation}\label{jieshi01}
  \nu_g(\dif s)=  h'(s)\cdot \1_{[a,b]}(s)\dif s+  h(s)\cdot \big(\delta_a(s)-\delta_b(s)\big)\dif s,
\end{equation}
where $\delta_a(\cdot)$ is the Dirac delta function centered at a point $a$. Correspondingly,
\begin{align}
\langle f,\,g \rangle_{\FH}&=H\int_{[0,T]}  f(t)  \dif t \int_{[a,b]}  h'(s) \abs{t-s}^{2H-1}\sgn(t-s) \dif s\notag\\
&+H\left[ \int_{[0,T]}  f(t)  \left[  h(a) \abs{t-a}^{2H-1}\sgn(t-a)  - h(b)  \abs{t-b}^{2H-1}\sgn(t-b)\right] \dif t\right].
\end{align}

\subsection{Complex Wiener-It\^o multiple integrals and complex fourth moment theorem}
In this subsection, we will introduce some essential definitions for understanding the complex Wiener-It\^o integrals and the complex fourth moment theorem. Later, the fourth moment theorem will be applied to prove our main result Theorem \ref{com-ou parameter estimate}. Throughout the context the notation $\FH$ is used for denoting a complex separable Hilbert space with inner product $\innp{\cdot,\cdot}_{\FH}$, linear in the first argument and conjugate linear in the second argument.
\begin{definition}
 Let $z = x+ \mi y$ with $x, y\in\Rnum$. Complex Hermite polynomials $J_{m,n}(z)$ are given by its generating function:
 \begin{align*}
     \exp\set{\lambda \bar{z}+\bar{\lambda} z -2\abs{\lambda}^2}=\sum_{m=0}^{\infty}\sum_{n=0}^{\infty}\frac{\bar{\lambda}^m\lambda^n}{m!n!}J_{m,n}(z).
 \end{align*}
\end{definition} 
It is clear that the complex Hermite polynomials satisfy
\begin{align}\label{dongegx-key}
    \overline{J_{m,n}(z)}=J_{n,m}(z).
\end{align}
\begin{definition}
    A complex Gaussian isonormal process $\set{Z(h):\,h\in \FH}$ over the complex Hilbert space $\FH$, is a centered symmetric complex Gaussian family in $L^2(\Omega)$ such that 
\begin{align*}
    \E[Z(h)^2]=0,\quad \E[Z(g)\overline{Z(h)}]=\innp{g, h}_{\FH}, \quad \forall g, h\in \FH. 
\end{align*}
\end{definition}

Now we can construct a special subspace of $L^2(\Omega)$, known as Wiener-It\^o chaos.

\begin{definition}\label{def.complex.ito}
    For each $m, n \ge 0$, let $\mathcal{H}_{m,n}$ indicate the closed linear subspace of $L^2(\Omega)$ generated by the random variables of the type    \begin{align*}
    \set{J_{m,n}(Z(h)):\, h\in \FH, \, \norm{h}_{\FH}=\sqrt{2}}.
\end{align*}
The space $\mathcal{H}_{m,n}$ is called 
the $(m,n)$-th Wiener-It\^o chaos of $Z$.
\end{definition} 

Next, we would like to define the complex Wiener-It\^o integrals using the linear isometry between the Hilbert space and the Wiener-It\^o chaos. To start, we denote by $\FH^{\otimes n}$ and $\FH^{\odot n}$ the $n$th tensor product and the $n$th symmetric tensor product of $\FH$ respectively. For example, for $n=2$, let $\mathcal{E}$ be the set of all finite linear combinations of the conjugate bilinear forms $ h_1\otimes h_2$ over $\FH\times \FH$:
\begin{equation*}
   [h_1\otimes h_2](x,y)=\innp{h_1,x}_{\FH}\innp{h_2,y}_{\FH},\quad (x,y)\in  \FH\times \FH.
\end{equation*}
We define an inner product $\innp{\cdot,\cdot}_{\mathcal{E}} $ on $\mathcal{E}$ by defining
\begin{equation*}
   \innp{h_1\otimes h_2,\,f_1\otimes f_2}_{\mathcal{E}}=\innp{h_1, f_1}_{\FH}\innp{h_2, f_2}_{\FH}
\end{equation*}and extending by linearity to $\mathcal{E}$.
The tensor product $\FH^{\otimes 2}$ of $\FH$ is defined as the completion of $\mathcal{E}$ under the inner product $\innp{\cdot,\cdot}_{\mathcal{E}}$. The symmetric tensor product $h_1\odot h_2$ is defined as $\frac12(h_1\otimes h_2+h_2\otimes h_1)$ (See \cite{simon}). 

\begin{definition}\label{chjfendy}
    For each $m, n \ge 0$, the linear mapping $$I_{m,n}(h^{\otimes m}\otimes \bar{h}^{\otimes n})=J_{m,n}(Z(h)),\quad h \in \mathfrak{H}$$ is called the complex Wiener-It\^o stochastic integral. The mapping $I_{m,n}$ provides a linear isometry between $\mathfrak{H}^{\odot m}\otimes \mathfrak{H}^{\odot n}$ (equipped with the norm $\frac{1}{\sqrt{m!n!}}\|\cdot\|_{\mathfrak{H}^{\otimes (m+n)}}$) and $\mathcal{H}_{m,n}$. Here $\mathcal{H}_{0,0} = \mathbb{R}$ and $I_{0,0}(x)=x$ by convention.
\end{definition}
From the above definition and the identity \eqref{dongegx-key}, if $f\in \mathfrak{H}^{\odot m}\otimes \mathfrak{H}^{\odot n}$ and $g\in \mathfrak{H}^{\odot n}\otimes \mathfrak{H}^{\odot m}$ satisfies a conjugate symmetry relation
\begin{align}\label{chongjifengongegx}
    g(t_1,\dots, t_n; s_1,\dots, s_m)=\overline{f(s_1,\dots, s_m; t_1,\dots, t_n )},
\end{align}
their complex Wiener-It\^o integrals must satisfy
\begin{align*}
    \overline{I_{m,n}(f)}=I_{n,m}(g).
\end{align*} 

The last concept introduced in this section would be the contraction of two functionals in the tensor product of Hilbert space $\mathfrak{H}$.
\begin{definition}
When $\FH=L^2(A,\mathcal{B}, \nu)$ with $\nu$ non-atomic, the $(i, j)$ contraction of two  symmetric functions $f\in \mathfrak{H}^{\odot a}\otimes \mathfrak{H}^{\odot b},\, g\in \mathfrak{H}^{\odot c}\otimes \mathfrak{H}^{\odot d} $ is defined as
\begin{align*}
  &\quad f\otimes_{i,j} g (t_1,\dots, t_{a+c-i-j};s_1,\dots, s_{b+d-i-j})\\
  &=\int_{A^{i+j}}\nu^{i+j}(\dif u_1\dots \dif u_i\dif v_1\dots \dif v_j)f(t_1,\dots, t_{a-i}, u_1, \dots , u_i; s_1,\dots, s_{b-j}, v_1,\dots, v_j)\\
  &\times g(t_{a-i+1},\dots, t_{a+c-i-j}, v_1,\dots, v_j; s_{b-j+1},\dots, s_{b+d-i-j}, u_1, \dots , u_i).
\end{align*} 
\end{definition}

The following complex fourth moment theorem simplifies the one in \cite{chw} (see \cite{withhuiping}).
\begin{theorem}[Fourth Moment Theorem]\label{com FMT}
    Let $\set{F_k=I_{m,n}(f_k)}$ be a sequence of $(m, n)$-th complex Wiener-It\^o integrals, where $m, n$ are fixed and $m + n \ge 2$. Suppose that as $k\to\infty$, $\E[\abs{F_k}^2]\to \sigma^2$ and $\E[{F^2_k}]\to c+\mi b$, where $\abs{\cdot}$ is the absolute value (or modulus) of a complex number and $c, b\in \Rnum$. Then the following statements are equivalent:
   \begin{description}
   \item [(i)] The sequence $(\RE F_k, \IM  F_k)$ converges in law to a bivariate normal distribution with variance-covariance matrix $\tensor{C}=\frac12
   \begin{bmatrix}
       \sigma^2+c, & b\\
       b,&\sigma^2-c
   \end{bmatrix}$.
    \item [(ii)] $\E[\abs{F_k}^4]\to c^2+b^2+2\sigma^4$.
    \item [(iii)] $\norm{f_k\otimes_{i,j}h_k}_{\FH^{\otimes(2(l-i-j))}}\to 0$ for any $0 <
i+j \le l-1$ where $l= m+n$ and $h_k$ is the kernel of $\bar{F}_k$, i.e., $\bar{F}_k=I_{n,m}(h_k)$. 
   \end{description}
\end{theorem}

\section{The second moments of two double Wiener-It\^{o} Integrals}\label{sec.3}

The goal of this section is to prove Proposition~\ref{qq1a} for computing the second moments of the double Wiener-It\^o integrals. 
We use $\delta_a(\cdot)$ to denote the Dirac delta function centered at point $a$. The Heaviside step function $\mathrm{H}(x)$ is defined as
\begin{equation*}
\mathrm{H}(x)=\left\{
      \begin{array}{ll}
 1, & \quad \text{if } x\ge 0,\\
0, &\quad \text{if } x<0.
     \end{array}
\right.
\end{equation*}
 The distributional derivative of the Heaviside step function is the Dirac delta function: $$\frac{\dif \mathrm{H}(x)}{\dif x}=\delta_0(x).$$
 Hence, for any $-\infty<a<b<\infty$,
 \begin{align}\label{distribution deriva}
 \frac{\dif }{\dif x} \mathbf{1}_{[a,b)}(x)= \frac{\dif }{\dif x} [\mathrm{H}(x-a)-\mathrm{H}(x-b)]=\delta_a(x)-\delta_b(x).
 \end{align}
 This fact implies that for the multivariable function $q(t,s):= \mathbf{1}_{\{0\le s\le t\le T\}} $,
\begin{equation}
\frac{\partial}{\partial t}q(t,s)=\mathbf{1}_{[0,T]}(s)\big( \delta_s(t) -\delta_T(t)\big), \quad \frac{\partial}{\partial s}q(t,s)=\mathbf{1}_{[0,T]}(t)\big(\delta_0(s) - \delta_t(s) \big).
\label{qstdaoshu}
 \end{equation}Clearly, these partial derivatives should be understood as distributional derivatives.
Similarly, for the function $p(u,v):= \mathbf{1}_{ (s,t)}(u)\mathbf{1}_{ (0,u)}(v) $, we have the distributional derivatives
\begin{equation}
\frac{\partial}{\partial u}p(u,v)
= \mathbf{1}_{ (0,t)}(v)\big(\delta_{v\vee s}(u)-\delta_t(u)\big) , \quad
\frac{\partial}{\partial v}p(u,v)=\mathbf{1}_{ (s,t)}(u)\big(\delta_0(v) - \delta_u(v) \big), \label{why-qstdaoshu-2-000}
\end{equation}
where $0\le s< t\le T$ are fixed. 

Next we will apply these facts to prove the inequalities in Proposition~\ref{qq1a}. The proposition will be used to validate the conditions in the fourth moment theorem when we prove 
Thoerem~\ref{com-ou parameter estimate} in the next section. The technical results about the asymptotic behaviour of some special integrals that are used in the proof are provided in Section~\ref{appendix}.\\

\textbf{Proof of Proposition}~\ref{qq1a}: Denote $\beta=2H-1$ and $\dif \vec{t}=\dif {t_1}\dif {t_2},\,\dif \vec{s}=\dif {s_1}\dif {s_2}$. 
It follows from equation \eqref{qstdaoshu}, the distributional derivatives of the function $q(t, s)$, that 
 \begin{align}
 \|\psi_T\|^2_{{\FH}^{\otimes2}}&  
=H^2\int_{[0,T]^4} \frac{\partial^2}{\partial t_1 \partial s_2} \left\{e^{-\bar{\gamma}(t_1-s_1)-\gamma(t_2-s_2)} q(t_1,s_1)q(t_2,s_2)\right\}\notag\\
    &\times \sgn(t_2-t_1)  \abs{t_2-t_1}^{\beta}  \sgn(s_1-s_2)  \abs{s_1-s_2}^{\beta} \dif\vec{t}   \dif \vec{s}, \label{chufadian000-000}
    \end{align} where $\gamma=\lambda-\mi \omega$ and the second-order partial derivative (understood as distributional derivative) is given by:
\begin{align*}
    &\frac{\partial^2}{\partial t_1 \partial s_2} \left\{e^{-\bar{\gamma}(t_1-s_1)-\gamma(t_2-s_2)} q(t_1,s_1)q(t_2,s_2)\right\} \\
    &=e^{-\bar{\gamma}(t_1-s_1)-\gamma(t_2-s_2)} \left[-\bar{\gamma} q(t_1,s_1)+ \mathbf{1}_{[0,T]}(s_1)\big( \delta_{s_1}(t_1) -\delta_T(t_1)\big)\right]\\
    &\times \left[{\gamma} q(t_2,s_2)+ \mathbf{1}_{[0,T]}(t_2)\big( \delta_{0}(s_2) -\delta_{t_2}(s_2)\big) \right].
\end{align*}
Expanding the equation \eqref{chufadian000-000}, we have 
\begin{align}    \|\psi_T\|^2_{{\FH}^{\otimes2}}=H^2\times [\mathrm{A}_1(T)+\mathrm{A}_{2}(T)+\mathrm{A}_3(T)],\label{qishift2-fenjie-000}
\end{align} where
\begin{align*}
    \mathrm{A}_1(T)
    &=-\abs{\gamma}^2\int_{[0,T]^4}  e^{-\bar{\gamma}(t_1-s_1)-\gamma(t_2-s_2)}q(t_1,s_1)q(t_2,s_2)\\
    &\times \sgn(t_2-t_1)  \abs{t_2-t_1}^{\beta}  \sgn(s_1-s_2)  \abs{s_1-s_2}^{\beta } \dif \vec{t}  \dif \vec{s},\\  
        \mathrm{A}_{2}(T)
    &=\int_{[0,T]^4}  e^{-\bar{\gamma}(t_1-s_1)-\gamma(t_2-s_2)}\sgn(t_2-t_1)  \abs{t_2-t_1}^{\beta}  \sgn(s_1-s_2)  \abs{s_1-s_2}^{\beta }\notag\\
    &\times \Big[-\bar{\gamma}q(t_1,s_1) \big( \delta_{0}(s_2) -\delta_{t_2}(s_2)\big)+\gamma q(t_2,s_2)\big( \delta_{s_1}(t_1)\big) -\delta_T(t_1)\big)\Big] \dif \vec{t}  \dif \vec{s},\\
   \mathrm{A}_3(T)&= \int_{[0,T]^4}  e^{-\bar{\gamma}(t_1-s_1)-\gamma(t_2-s_2)}\sgn(t_2-t_1)  \abs{t_2-t_1}^{\beta}  \sgn(s_1-s_2)  \abs{s_1-s_2}^{\beta }\notag\\
    &\times  \big( \delta_{s_1}(t_1) -\delta_T(t_1)\big) \big( \delta_{0}(s_2) -\delta_{t_2}(s_2) \big) \dif \vec{t}  \dif \vec{s}. 
\end{align*}

 The term $\mathrm{A}_1(T)$ can be decomposed into conjugate integrals over two sub-domains $t_2\le t_1$ and $t_1\le t_2$, 
\begin{align}
    \mathrm{A}_1(T)&=2{\abs{\gamma}^2} \RE\Big[ \int_{[0,T]^4, t_2\le t_1}  e^{-\bar{\gamma}(t_1-s_1)-\gamma(t_2-s_2)}q(t_1,s_1)q(t_2,s_2)\notag \\
    &\times (t_1-t_2)^{\beta}  \sgn(s_1-s_2)  \abs{s_1-s_2}^{\beta } \dif \vec{t}  \dif \vec{s} \Big].\notag\\
\end{align}
By Lemma~\ref{last key 00}, for $H\in (\frac14,\frac12)$, we have
\begin{equation}
   \mathrm{A}_1(T) = \frac{T^{4H}}{2H(4H-1) }-\frac{T\abs\gamma^2}{\lambda}\left[ \Gamma^2(2H) - 2\kappa \right] \RE\big(\frac{1}{\gamma^{4H}}\big) -\frac{2\lambda}{(4H-1)\abs{\gamma}^2}T^{4H-1} +O(1),\label{fenjie 0012-000}
\end{equation}
where 
\begin{equation}
     \kappa=-\frac{\Gamma(2H)\Gamma(4H-1)\Gamma(3-4H)}{2\Gamma(2-2H)}.
     \label{kapp notation}
\end{equation}

Similarly, for the term $A_2$, we first expand the integrand according to the definition of Dirac delta function, and utilize the conjugate relationships to 
obtain
\begin{align}
    A_2 &=2 \RE \left[\gamma\int_{[0,T]^3} e^{-\gamma (t_2-s_2)}\abs{t_2-s_1}^{\beta}\abs{s_1-s_2}^{\beta}\sgn(t_2-s_1)\sgn(s_1-s_2)q(t_2,s_2)\dif \vec{s}\dif t_2 \right.\notag \\
    &+ \left.\gamma\int_{[0,T]^3}e^{-\gamma (t_2-s_2)-\bar{\gamma}s_1}\abs{t_2-s_1}^{\beta}\sgn(t_2-s_1) s_2^{\beta} q(t_2,s_2) \dif \vec{s}\dif t_2\right].  \label{fenjie 002-000-buch-00}
\end{align}
For the first integral of the term $A_2$, we split the region of integration into two sub-regions: $\{s_2 \leq s_1 \leq t_2\}$, $\{s_1 \leq s_2\} \cup \{t_2 \leq s_1\}$. Then we have the following integration results:
\begin{itemize}
 \item Over the first sub-region, Making change of variables $  x=s_1-s_2, y=t_2-s_2$ yields
\begin{align*}
    &\int_{[0,T]^3,\,s_2\le  s_1\le t_2}e^{-\gamma (t_2-s_2)}\abs{t_2-s_1}^{\beta}\abs{s_1-s_2}^{\beta}\sgn(t_2-s_1)\sgn(s_1-s_2)\dif \vec{s}\dif t_2\\
    &=\int_{0\le x\le y\le T}e^{-\gamma y} x^{\beta}(y-x)^{\beta}(T-y)\dif x\dif y=\Gamma^2(2H){\gamma}^{-4H}T +O(1);
\end{align*}
 \item Over the second sub-region, Making change of variables $  x=\abs{s_1-s_2}\wedge \abs{t_2-s_1},  y=\abs{s_1-s_2}\vee \abs{t_2-s_1}$ and $z=s_1\wedge s_2$ yields
\begin{align*}
    &\int_{[0,T]^3,\,s_1\le  s_2, \text{\,or,\,}  t_2\le s_1}e^{-\gamma (t_2-s_2)}\abs{t_2-s_1}^{\beta}\abs{s_1-s_2}^{\beta}\sgn(t_2-s_1)\sgn(s_1-s_2)q(t_2,s_2)\dif \vec{s}\dif t_2\\
    &=-2\int_{0\le x\le y\le T}e^{\gamma(x-y)} x^{\beta}y^{\beta}(T-y)\dif x\dif y.
\end{align*}
\end{itemize}
For the second integral of the term $A_2$, we split the region of integration into two sub-regions: $\{s_1 \leq s_2\}$ and $\{s_1 > s_2\}$. Then we have:
\begin{itemize}
 \item Over the first sub-region, making change of variables $z=s_2-s_1, x=s_2,y=t_2-s_1$ and using the symmetry yield 
\begin{align*}
   & \RE\big[\gamma\int_{[0,T]^3,\,s_1\le s_2}e^{-\gamma (t_2-s_2)-\bar{\gamma}s_1}\abs{t_2-s_1}^{\beta}\sgn(t_2-s_1) s_2^{\beta} q(t_2,s_2) \dif \vec{s}\dif t_2\Big]\\
    &= \RE\big[\gamma \int_{[0,T]^2}e^{-\gamma y -\bar{\gamma} x}x^{\beta}y^{\beta}\dif x\dif y \int_{0\vee (x+y-T)}^{x\wedge y} e^{2\lambda z }\dif z\big]\\
    &= \RE\big[\frac{\gamma}{2\lambda} \int_{[0,T]^2}e^{-\gamma y-\bar{\gamma} x+2\lambda (x\wedge y )} x^{\beta}y^{\beta} \dif x\dif y+O(1)\big],\quad \text{(by Lemma~\ref{upper bound F})} \\
&=\frac12 \int_{[0,T]^2}e^{-\gamma y-\bar{\gamma} x+2\lambda (x\wedge y )} x^{\beta}y^{\beta} \dif x\dif y+O(1)\\
  &=\RE \int_{0\le x\le y\le T}e^{\gamma (x-y)} x^{\beta}y^{\beta} \dif x\dif y+O(1),
\end{align*}
where the double integral in the third step is actually real-valued.
 \item Over the second sub-region, making change of variables $  x=s_2,y=\abs{t_2-s_1}, z=(t_2\wedge s_1)-s_2$ yields
\begin{align*}
    \int_{[0,T]^3,\,s_1> s_2}e^{-\gamma (t_2-s_2)-\bar{\gamma}s_1}\abs{t_2-s_1}^{\beta}\sgn(t_2-s_1) s_2^{\beta} q(t_2,s_2) \dif \vec{s}\dif t_2=O(1).
\end{align*}
\end{itemize}

We add the four integration results from above to obtain
  \begin{align*}
    A_2&={2 }\RE\Bigg[\gamma\Big(\Gamma^2(2H){\gamma}^{-4H}T +O(1)-2 \int_{0\le x\le y\le T}e^{\gamma(x-y)} x^{\beta}y^{\beta}(T-y)\dif x\dif y\Big)\\
    &+\int_{0\le x\le y\le T}e^{\gamma (x-y)} x^{\beta}y^{\beta} \dif x\dif y  +O(1)\Bigg]\\
    &={2 }\RE\Bigg[ \int_{0 \leq x \leq y \leq T} e^{\gamma(x-y)}x^{\beta}y^{\beta}\Big[ {1}-2\gamma(T-y) \Big]\dif x\dif y+\Gamma^2(2H){\gamma}^{1-4H}T\Bigg]+O(1).\notag
  \end{align*}

Next we can apply  equation \eqref{zhjjg 0000} to obtain for $H\in (\frac14,\frac12)$,
\begin{align}
   A_2 &={2 }\RE\Bigg[\Gamma^2(2H){\gamma}^{1-4H}T+\frac{T^{4H}}{2H}-\frac{2T^{4H-1}}{\gamma}\notag\\
    &+(1+4H-2\gamma T)\int_{0\le x\le y\le T} e^{\gamma(x-y)}x^{\beta}y^{\beta} \dif x\dif y\Bigg]+O(1). \notag
\end{align}
Applying Lemma~~\ref{asymptotic expansion key} for $\delta=1+2\beta=4H-1 \in (0, 1)$ yields
   \begin{align}
    A_2 & = {2}\RE\Bigg[\Gamma^2(2H){\gamma}^{1-4H}T+\frac{T^{4H}}{2H}-\frac{2T^{4H-1}}{\gamma} \notag\\
    & (1+4H-2\gamma T) \left(\frac{1}{4H-1}\frac{T^{4H-1}}{\gamma}+ \frac{\kappa}{\gamma^{4H}}-\frac{T^{4H-2}}{2\gamma^2}+O(T^{4H-3})\right)\Bigg]+O(1),\notag
   \end{align}
   where $\kappa$ is given by \eqref{kapp notation}. We simplify the above expression to obtain
   \begin{align}
    A_2 &=\left(\Gamma^2(2H)-2\kappa\right)\left(\gamma^{1-4H}+ \bar{\gamma}^{1-4H}\right)T \notag\\
    &-\frac{1}{H(4H-1)}T^{4H}+ \frac{2}{4H-1}\left(\frac{1}{\gamma}+\frac{1}{\bar{\gamma}}\right)T^{4H-1}+O(1). \label{fenjie 0023-000}
  \end{align}

For the last term $A_3$, by the definition of Dirac delta function, we simplify
\begin{align}  
    A_3 &=\int_{[0,T]^2} e^{-\bar{\gamma}(T-s_1)-\gamma t_2} (T-t_2)^{\beta} s_1^{\beta}\dif t_2\dif s_1  +\int_{[0,T]^2} \abs{t_2-s_1}^{2\beta} \dif t_2\dif s_1 \notag \\
    &+ \int_{[0,T]^2} e^{-\gamma t_2} \abs{t_2-s_1}^{\beta}\sgn(t_2-s_1) s_1^{\beta}\dif t_2\dif s_1 \notag \\
    &-\int_{[0,T]^2} e^{-\bar{\gamma}(T-s_1) } (T-t_2)^{\beta} \abs{t_2-s_1}^{\beta}\sgn(s_1-t_2)\dif t_2\dif s_1.  \label{fenjie 003-000-buch}
\end{align}
For the first integral, Lemma~\ref{upper bound F} implies that 
\begin{align*}
    \int_{[0,T]^2} e^{-\bar{\gamma}(T-s_1)-\gamma t_2} (T-t_2)^{\beta} s_1^{\beta}\dif t_2\dif s_1 =O(1).
\end{align*}
For the second integral, the symmetry implies that 
\begin{align*}
    \int_{[0,T]^2} \abs{t_2-s_1}^{2\beta} \dif t_2\dif s_1 
    = \left\{
      \begin{array}{ll}
 +\infty, & \quad  H \in (0, \frac14],\\
\frac{T^{4H}}{2H(4H-1)} ,& \quad H \in (\frac14,\frac12).
 \end{array}
\right.   
\end{align*}
For the fourth integral, making the change of variables $u=T-s_1,\, v=T-t_2$, we have
\begin{align*}
\int_{[0,T]^2} e^{-\bar{\gamma}(T-s_1) } (T-t_2)^{\beta} \abs{t_2-s_1}^{\beta}\sgn(s_1-t_2)\dif t_2\dif s_1 = \int_{[0,T]^2} e^{-\bar{\gamma} u}v^{\beta} \abs{v-u}^{\beta}\sgn(v-u) \dif u\dif v.
\end{align*}
The third integral is conjugate to the fourth integral. Hence, substituting the above equations into \eqref{fenjie 003-000-buch} 
we have that when  $H\in (\frac14, \frac12)$,
\begin{align}
    A_3  &=O(1)+\frac{T^{4H}}{2H(4H-1)}-2\RE \int_{[0,T]^2} e^{- {\gamma} u}v^{\beta} \abs{v-u}^{\beta}\sgn(v-u) \dif u\dif v\notag \\
    &=\frac{T^{4H}}{2H(4H-1)}-2\RE\int_{0\le x\le y\le T}e^{{\gamma} (x-y)}x^{\beta}y^{\beta} \dif x\dif y +O(1)\notag\\
    &=\frac{T^{4H}}{2H(4H-1)}- \left(\frac{1}{\gamma}+\frac{1}{\bar{\gamma}}\right)\frac{T^{4H-1}}{4H-1} +O(1),\label{fenjie 003-000}
\end{align}where in the second line we have made the change of variables $x=\abs{v-u},\, y= u\vee v$, and in the last line, we have used Lemma~~\ref{asymptotic expansion key}. 

Substituting these integration results \eqref{fenjie 0012-000}, \eqref{fenjie 0023-000}, \eqref{fenjie 003-000} into equation \eqref{qishift2-fenjie-000},  we obtain the inequality 
\eqref{diyigedingli-complex}. 

As the last part of this section, we sketch the proof of the inequality 
\eqref{diyigedingli-complex-002} 
briefly. It follows from equation \eqref{qstdaoshu} that 
 \begin{align}
 \innp{\psi_T, h_T}_{{\FH}^{\otimes2}}&  
=H^2\int_{[0,T]^4} \frac{\partial^2}{\partial t_1 \partial s_2} \left\{e^{-\bar{\gamma}(t_1-s_1)-\bar{\gamma}(s_2-t_2)} q(t_1,s_1)q(s_2,t_2)\right\}\notag\\
    &\times \sgn(t_2-t_1)  \abs{t_2-t_1}^{\beta}  \sgn(s_1-s_2)  \abs{s_1-s_2}^{\beta} \dif\vec{t}   \dif \vec{s}, \label{chufadian000-000-001}
    \end{align}where the second-order partial derivative  (understood as distributional derivative) is given by:
\begin{align*}
    &\frac{\partial^2}{\partial t_1 \partial s_2} \left\{e^{-\bar{\gamma}(t_1-s_1)-\bar{\gamma}(s_2-t_2)} q(t_1,s_1)q(s_2,t_2)\right\} \\
    &=e^{-\bar{\gamma}(t_1-s_1)-\bar{\gamma}(s_2-t_2)} \left[-\bar{\gamma} q(t_1,s_1)+ \mathbf{1}_{[0,T]}(s_1)\big( \delta_{s_1}(t_1) -\delta_T(t_1)\big)\right]\\
    &\times \left[-\bar{\gamma} q(s_2,t_2)+ \mathbf{1}_{[0,T]}(t_2)\big( \delta_{t_2}(s_2) -\delta_{T}(s_2)\big) \right].
\end{align*}
Expanding equation \eqref{chufadian000-000-001}, we have 
\begin{align}     \innp{\psi_T, h_T}_{{\FH}^{\otimes2}}=H^2\times [\mathrm{K}_1(T)+\mathrm{K}_{2}(T)+\mathrm{K}_3(T)],\label{qishift2-fenjie-000-001}
\end{align} where  
\begin{align*}
    \mathrm{K}_1(T)
    &=\bar{\gamma}^2\int_{[0,T]^4}  e^{-\bar{\gamma}(t_1-s_1)-\bar{\gamma}(s_2-t_2)}q(t_1,s_1)q(s_2,t_2)\notag\\
    & \quad \times \sgn(t_2-t_1)  \abs{t_2-t_1}^{\beta}  \sgn(s_1-s_2)  \abs{s_1-s_2}^{\beta } \dif \vec{t}  \dif \vec{s}, \notag \\
    \mathrm{K}_{2}(T)
    &=-\bar{\gamma}\int_{[0,T]^4}  e^{-\bar{\gamma}(t_1-s_1)-\bar{\gamma}(s_2-t_2)}\sgn(t_2-t_1)  \abs{t_2-t_1}^{\beta}  \sgn(s_1-s_2)  \abs{s_1-s_2}^{\beta }\notag\\
    & \quad \times \Big[q(t_1,s_1) \big( \delta_{t_2}(s_2) -\delta_{T}(s_2)\big)+  q(s_2,t_2)\big( \delta_{s_1}(t_1)\big) -\delta_T(t_1)\big)\Big] \dif \vec{t}  \dif \vec{s}, \notag\\
    \mathrm{K}_3(T)
    &= \int_{[0,T]^4}  e^{-\bar{\gamma}(t_1-s_1)-\bar{\gamma}(s_2-t_2)} \sgn(t_2-t_1) \abs{t_2-t_1}^{\beta}  \sgn(s_1-s_2)  \abs{s_1-s_2}^{\beta }\notag\\
    & \quad \times \big( \delta_{s_1}(t_1) -\delta_T(t_1)\big)\big( \delta_{t_2}(s_2) -\delta_{T}(s_2)\big) \dif \vec{t}  \dif \vec{s}. \notag
\end{align*}
For the term $\mathrm{K}_1(T)$, we first use symmetry to write the term as
\begin{align*}
    & \mathrm{K}_1(T) =
    2\bar{\gamma}^2\int_{[0,T]^4, t_1\le s_2}  e^{-\bar{\gamma}(t_1-s_1)-\bar{\gamma}(s_2-t_2)}q(t_1,s_1)q(s_2,t_2)\notag\\
    & \quad \times \sgn(t_2-t_1)  \abs{t_2-t_1}^{\beta}  \sgn(s_1-s_2)  \abs{s_1-s_2}^{\beta } \dif \vec{t}  \dif \vec{s}.\end{align*}
By Lemma~\ref{last key 00}, for $H\in (\frac14,\frac12)$, we have
\begin{align}
  \mathrm{K}_1(T) &=\frac{4H(\Gamma^2(2H)-2\kappa)}{\bar{\gamma}^{4H-1}}T-\frac{1}{2H(4H-1)}T^{4H} + \frac{2}{(4H-1)\bar{\gamma}}T^{4H-1} +O(1). \label{fenjie 0012-000-001}
\end{align}

For the term $K_2(T)$, using symmetry we express it as
\begin{align}
    \mathrm{K}_2(T)
    &=-2\bar{\gamma}\int_{[0,T]^4}  e^{-\bar{\gamma}(t_1-s_1)-\bar{\gamma}(s_2-t_2)}\sgn(t_2-t_1)  \abs{t_2-t_1}^{\beta}  \sgn(s_1-s_2)  \abs{s_1-s_2}^{\beta }\notag\\
    & \quad \times q(t_1,s_1) \big( \delta_{t_2}(s_2) -\delta_{T}(s_2)\big)  \dif \vec{t}  \dif \vec{s} \notag \\
    & = -2\bar{\gamma}\int_{[0,T]^3}  e^{-\bar{\gamma}(t_1-s_1)}\sgn(t_2-t_1)  \abs{t_2-t_1}^{\beta}  \sgn(s_1-t_2)  \abs{s_1-t_2}^{\beta}
    q(t_1, s_1) \dif \vec{t}  \dif s_1 \notag \\
    & - 2\bar{\gamma}\int_{[0,T]^3}  e^{-\bar{\gamma}(t_1-s_1)-\bar{\gamma}(T-t_2)}\sgn(t_2-t_1)  \abs{t_2-t_1}^{\beta} (T - s_1)^{\beta} q(t_1, s_1) \dif \vec{t}  \dif s_1.
    \label{2jifen fenj}
\end{align}
By the change of variables $x=\abs{t_1-t_2}\wedge \abs{s_1-t_2}, y=\abs{t_1-t_2}\vee \abs{s_1-t_2}$
, we have 
\begin{align}
& \int_{[0,T]^3}  e^{-\bar{\gamma}(t_1-s_1)}\sgn(t_2-t_1)  \abs{t_2-t_1}^{\beta}  \sgn(s_1-t_2)  \abs{s_1-t_2}^{\beta}
    q(t_1, s_1) \dif \vec{t}  \dif s_1 \notag\\
    &=\int_{0\le x+y\le T} e^{-\bar{\gamma}(x+y)}x^{\beta}y^{\beta}(T-x-y)\dif x\dif y-2\int_{0\le x\le y\le T}e^{\bar{\gamma}(x-y)}x^{\beta}y^{\beta}(T-y)\dif x\dif y.\label{k21 t js}
\end{align}
By the change of variables $x=\abs{t_1-t_2}, y=T-s_1, u=(t_1\vee t_2)-s_1$
, we have
\begin{align}
 & \int_{[0,T]^3}  e^{-\bar{\gamma}(t_1-s_1)-\bar{\gamma}(T-t_2)}\sgn(t_2-t_1)  \abs{t_2-t_1}^{\beta} (T - s_1)^{\beta} q(t_1, s_1) \dif \vec{t}  \dif s_1\notag\\
  &=-\int_{[0,T]^2}e^{-\bar{\gamma}(x+y)}x^{\beta}y^{\beta} (y-x)\dif x\dif y +\int_{[0,T]^2}e^{-\bar{\gamma}(x-y)}x^{\beta}y^{\beta} (y-x)\dif x\dif y\notag\\
  &-\int_{0\le x\le y\le T}e^{-\bar{\gamma}(x+y)}x^{\beta}y^{\beta}\big[x\wedge y -(x+y-T)\vee 0\big]\dif x\dif y
.  \label{k22 t js}
\end{align}
Plugging the integration results \eqref{k21 t js} and \eqref{k22 t js} into equation \eqref{2jifen fenj} yields
\begin{align*}
\mathrm{K}_2(T)
    &= {2 }\bar{\gamma} \left[\int_{0\le x\le y\le T} e^{\bar\gamma(x-y)}x^{\beta}y^{\beta}\big( 2T-3y+x \big)\dif x\dif y -T\int_{[0, T]^2} e^{-\bar{\gamma}(x+y)}x^{\beta}y^{\beta} dxdy \right]+O(1).
\end{align*} Applying Lemma~\ref{asymptotic expansion key} and Lemma~\ref{coro key point000} yields when $H\in (\frac{1}{4}, \frac12)$
\begin{align}
    \mathrm{K}_2(T)
    &= {2 }\bar{\gamma} \left[ 2T(\frac{T^{4H-1}}{(4H-1)\bar{\gamma}} +\frac{\kappa}{\bar{\gamma}^{4H}}-\frac{T^{4H-2}}{2\bar{\gamma}^2}) -\frac{T^{4H}}{2H\bar{\gamma}}+\frac{(4H-3)T^{4H-1}}{(4H-1)\bar{\gamma}^2} -\frac{\Gamma^2(2H)}{\bar{\gamma}^{4H}}T\right] +O(1)\notag\\
    &=\frac{1}{(4H-1)H}T^{4H}+\frac{2(2\kappa-\Gamma^2(2H))}{\bar{\gamma}^{4H-1}}T - \frac{4}{(4H-1)\bar{\gamma}}T^{4H-1}+O(1).\label{fenjie 0023-000-001}
\end{align}
For the term $K_3(T)$, by the definition of Dirac delta function and the symmetry, we simplify it as
\begin{align*}
  \mathrm{K}_3(T)&=  2\int_{[0,T]^2}e^{\bar{\gamma}(s_1-T)}\sgn(s_1-t_2)\abs{s_1-t_2}^{\beta}(T-t_2)^{\beta}\dif s_1\dif t_2-\int_{[0,T]^2} \abs{s_1-t_2}^{2\beta}\dif s_1\dif t_2\\
  &+\int_{[0,T]^2}e^{\bar{\gamma}(s_1-T+t_2-T)} (T-s_1)^{\beta}(T-t_2)^{\beta}\dif s_1\dif t_2.
\end{align*}It is clear the third integral is $O(1)$. For the first integral, 
by the change of variables $x=\abs{t_2-s_1},\, y=T-t_2 $, we obtain
\begin{align*}
& \int_{[0,T]^2}e^{\bar{\gamma}(s_1-T)}\sgn(s_1-t_2)\abs{s_1-t_2}^{\beta}(T-t_2)^{\beta}\dif s_1\dif t_2\\
&=\int_{0\le x\le y\le T} e^{\bar{\gamma}(x-y)}x^{\beta}y^{\beta}\dif x\dif y-\int_{0\le x+y\le T} e^{-\bar{\gamma}(x+y)}x^{\beta}y^{\beta}\dif x\dif y \\
&=\int_{0\le x\le y\le T} e^{\bar{\gamma}(x-y)}x^{\beta}y^{\beta}\dif x\dif y+O(1).
\end{align*} 
Therefore, we use the symmetry and Lemma~\ref{asymptotic expansion key} to obtain
\begin{align}
    \mathrm{K}_3(T)
    &=  2\left[ \int_{0\le x\le y\le T}e^{\bar{\gamma} (x-y)}x^{\beta}y^{\beta} \dif x\dif y - \int_{0\le t_2\le s_1\le T} (s_1-t_2)^{2\beta} \dif t_2\dif s_1\right]+O(1)\label{shuom}\\
     &=-\frac{1}{2H(4H-1)}T^{4H}+  \frac{2}{(4H-1)\bar{\gamma}} {T^{4H-1}} +O(1),\quad \text{when $H\in \left(\frac14, \frac12\right)$.} \label{fenjie 003-000-001}
\end{align}
The second integral in equation \eqref{shuom} is finite when $H\in (\frac14, \frac12)$ and is infinite when $H\in (0,\frac14]$. Plugging the above results \eqref{fenjie 0012-000-001}, \eqref{fenjie 0023-000-001} and \eqref{fenjie 003-000-001} into equation \eqref{qishift2-fenjie-000-001},  we can obtain 
the inequality \eqref{diyigedingli-complex-002}.
This concludes the proof of Proposition~\ref{qq1a}.

\section{The statistical inference of the two-dimensional fractional Ornstein-Ulenbeck process with Hurst parameter $H\in (0,\frac12)$}\label{sec.4}
In this section, we will prove our main result Theorem~\ref{com-ou parameter estimate}. 
At first, we will derive some auxiliary results regarding the asymptotic behaviours of some specific functions. 
The constant $\gamma=\lambda-\mi \omega,\,\lambda>0,\omega\neq 0$ is recalled as it would be used extensively in this section.

\begin{lemma}\label{lm4-1}
For the complex-valued function 
\begin{equation}\label{gu}
    g(u):=e^{-\gamma u}\1_{[0,T]}(u),
\end{equation}
we have 
\begin{align}
 \lim_{T\to \infty}  \norm{g}_{\FH}^2= H\Gamma(2H)\frac{1}{2\lambda}\big[\bar{\gamma}^{1-2H}+{\gamma}^{1-2H}\big].
\end{align}
\end{lemma}
\begin{proof}
The inner product formula \eqref{innp fg3-00} implies that
\begin{align*}
 \lim_{T\to \infty}  \norm{g}_{\FH}^2&=H\lim_{T\to \infty} \iint_{[0,T]^2}  g(u)\overline{g'(v)}\abs{u-v}^{2H-1}\sgn(u-v)\dif u\dif v
\end{align*} 
where 
\begin{align*}
    g'(v)=g(v)\left[-\gamma \1_{[0,T]}(v)+\delta_0(v)-\delta_T(v)\right].
\end{align*}A direct calculation implies that 
\begin{align*}
 \lim_{T\to \infty}  \norm{g}_{\FH}^2&=H\Gamma(2H)\left[\frac{\bar{\gamma}}{2\lambda}(\bar{\gamma}^{-2H}-{\gamma}^{-2H})+ {\gamma}^{-2H}\right] \\
 &=H\Gamma(2H)\frac{1}{2\lambda}\left[\bar{\gamma}^{1-2H}+{\gamma}^{1-2H}\right].
\end{align*}
\end{proof}

Next, we will study the limiting behavior of the time-averaged value regarding the solution to equation \eqref{cp}.
\begin{lemma}\label{lem fenmuzixian}
Let $Z$ be the solution to equation \eqref{cp} with $Z_0=0$. As $T\to \infty$, we have that 
\begin{align}
    \frac{1}{T}\int_0^T \abs{Z_t}^2\dif t \to H\Gamma(2H)\frac{1}{2\lambda}\left[\bar{\gamma}^{1-2H}+{\gamma}^{1-2H}\right],\quad a.s.. \label{zt.lim}
\end{align}
\end{lemma}
\begin{proof}
    Denote $Y_t=\int_{-\infty}^t e^{-\gamma (t-u)}\dif \zeta_u,\,t\in \Rnum$.  Clearly it is a centered Gaussian process. We claim that $Y$ is stationary and ergodic. By the stationarity of increments of fBm and the construction of stochastic integrals, we rewrite $Y_{t+s} = \int_{-s}^\infty e^{-\gamma(u+s)} d \zeta_{t-u}$ and $\bar{Y}_t = \int_0^\infty e^{-\bar\gamma v} d \zeta_{t-v}$, and apply It\^{o}'s isometry to obtain 
    \begin{align*}
        &\E[Y_{t+s}\bar{Y}_t]\\& = H e^{-\gamma s} \lim_{T\to\infty} \int_{-s}^{t+T}\dif u\int_{0}^{t+T} e^{-\gamma u} (e^{-\bar{\gamma}v}\1_{[0,T+t]}(v))' \abs{u-v}^{2H-1}\sgn(u-v)\dif v\\
        &=H  {e^{-\gamma s}}\big[\int_{-s}^{\infty}e^{-\gamma u}u^{2H-1}\sgn(u)\dif u -\bar{\gamma}\int_{-s}^{\infty}\dif u\int_{0}^{\infty} e^{-\gamma u-\bar{\gamma}v}   \abs{u-v}^{2H-1}\sgn(u-v)\dif v\big]\\
        &=\E[Y_{s}\bar{Y}_0].
    \end{align*}
    This implies that the process $Y$ is stationary. Since the function $g(s)=e^{\gamma s}$ and $g'(s)$ are never zero  and the ratio $\frac{\abs{g'}}{\abs{g}'}$ is defined and bounded in the neighbourhood of $\infty$,  by L'H\^{o}pital's rule for complex-valued functions (see \cite{Carter}), we have that as $s\to\infty$, $\E[Y_{s}\bar{Y}_0]\to 0$ when $H\in (0,\frac12)$. This implies the complex-valued Gaussian process $Y$ is ergodic.
    
    Next, we will prove \eqref{zt.lim}. It is clear that $
    Z_t=Y_t- e^{-\gamma t} Y_0$. By the ergodic property of the process $Y$ and the Cauchy-Schwarz inequality, we can compute
\begin{align*}
\lim_{T\to\infty}   \frac{1}{T}\int_0^T \abs{Z_t}^2\dif t &=\lim_{T\to\infty}   \frac{1}{T}\int_0^T \abs{Y_t}^2\dif t=\E[\abs{Y_0}^2]\\ &=H\Gamma(2H)\frac{1}{2\lambda}\big[\bar{\gamma}^{1-2H}+{\gamma}^{1-2H}\big],
\end{align*}where the last line can be obtained from It\^o's isometry and Lemma~\ref{lm4-1}.
\end{proof}

In the following proposition, we will study the limiting behavior of the complex (1, 1) Wiener-It\^o integral process $X_t$ given by \eqref{e.def-X}.

\begin{proposition}\label{prop fenzi as converg}
Let the process $X_T$ be defined by \eqref{e.def-X}. When $H\in (\frac14,\frac12)$, the Wiener chaos process $\set{\frac{X_T}{T},\,T>0}$ converges to zero almost surely as $T\to\infty$.
\end{proposition} 
\begin{proof}
 The proof is similar to the case of $H\in[\frac12, \frac34)$ (see \cite{chw}). For the reader's convenience, we sketch the proof here.

When $H\in (\frac14,\frac12)$, it follows from the equation \eqref{sigmah0002},
Borel-Cantelli lemma and the
hypercontractivity of complex multiple Wiener-It\^o integrals that 
as $n\to\infty$ the sequence $\set{\frac{X_n}{n},\,n\ge 1}$ converges to zero almost surely.

Next, we can obtain the trajectory regularity of the increments of the process $\set{ {X_t} ,\,t\in [0,T]}$.
In fact, it follows from It\^o's isometry of complex multiple Wiener-It\^o integrals and  the identity \eqref{why-qstdaoshu-2-000} that 
 \begin{align}
\E[\abs{X_t-X_s} ^2]&  
=H^2\int_{[0,T]^4} \frac{\partial^2}{\partial u_1 \partial v_2} \left\{e^{-\bar{\gamma}(u_1-v_1)-{\gamma}(u_2-v_2)} p(u_1,v_1)p(u_2,v_2)\right\} \notag\\
    &\times \sgn(u_2-u_1)  \abs{u_2-u_1}^{\beta}  \sgn(v_1-v_2)  \abs{v_1-v_2}^{\beta} \dif\vec{u}   \dif \vec{v}, \label{chufadian000-000-001-zl}
    \end{align}where $\dif\vec{u}=\dif{u}_1\dif{u}_2,   \dif \vec{v} =\dif{v}_1\dif{v}_2$, and 
    the second-order partial derivative  (understood as distributional derivative)
    is given by:
\begin{align}
    &\frac{\partial^2}{\partial u_1 \partial v_2} \left\{e^{-\bar{\gamma}(u_1-v_1)-{\gamma}(u_2-v_2)} p(u_1,v_1)p(u_2,v_2)\right\} \notag\\
    &=e^{-\bar{\gamma}(u_1-v_1)-{\gamma}(u_2-v_2)} \left[{\gamma} p(u_2,v_2)+ \mathbf{1}_{(s,t)}(u_2)\big( \delta_{0}(v_2) -\delta_{u_2}(v_2)\big) \right] \notag\\
    &\times \left[-\bar{\gamma} p(u_1,v_1)+ \mathbf{1}_{(0,t)}(v_1)\big(\delta_{v_1\vee s}(u_1)-\delta_{t}(u_1) \big)\right].\label{chufadian000-000-001-zl-daoshu}
\end{align}
Using Lemma~\ref{upper bound F}, and after some similar calculations to the proof of Proposition~\ref{qq1a}, we can show that there exists a constant $C>0$ independent of $T$ such that for all $s,t\ge 0$ and $\abs{s-t}< 1$, 
\begin{align}\label{E norm 2} \E[\abs{X_t-X_s} ^2]\le C  \abs{t-s}^{2H}, \end{align} where we would like to point out that in the equations \eqref{chufadian000-000-001-zl}-\eqref{chufadian000-000-001-zl-daoshu}, there is a double integral which is finite only when $H\in (\frac14,\frac12)$ similar to that in both \eqref{fenjie 003-000} and \eqref{fenjie 003-000-001}.

Using the equation \eqref{E norm 2}, the hypercontractivity of multiple Wiener-It\^o integrals and the Garsia-Rodemich-Rumsey inequality, we can show that for any real number $p > \frac{2}{H}, q >1$ and integer $n \ge1$, 
\begin{align*}
\abs{X_t-X_s}\le R_{p,q} n^{q/p} , \qquad \forall \ t,s\in [n,n+1] ,
\end{align*}where $R_{p,q}$ is a random constant independent of $n$ (see \cite{chw}).

Finally, since
\begin{align*}
\abs{\frac{X_T}{T}}\le \frac{1}{T}\abs{X_T-X_n}+ \frac{n}{T}\frac{\abs{X_n}}{n},
\end{align*}where $n=[T]$ is the biggest integer less than or equal to a real number $T$, we have $\frac{X_T}{T}$ converges to $0$ almost surely as $T\to \infty$.
\end{proof}
The following proposition will play an important role when we validate the conditions of the fourth moment theorem to prove Thoerem~\ref{com-ou parameter estimate}.

\begin{proposition}\label{con equa 0}
Let $\psi_T,\,h_T$ be defined by \eqref{phist defn} and $H\in (\frac{1}{6},\,\frac12)$. As $T\to \infty$, we have
\begin{equation}
   \frac{1}{T}\psi_T{\otimes}_{0,1}h_T\to 0,\quad \frac{1}{T}\psi_T{\otimes}_{1,0}h_T\to 0,\quad \text{in }\,  \FH^{\otimes 2}.
\end{equation}
\end{proposition}
\begin{proof} It suffices to show $\phi(s,t):=\frac{1}{T}\psi_T{\otimes}_{1,0}h_T\to 0$, as the other one is similar. The proof proceeds along similar lines to the equation (3.17) in \cite{hnz 19}. In fact, the inner product formula of \cite{pip} implies that 
\begin{align}
   \phi(s,\,t)&=\frac{1}{T}\innp{\psi_T(\cdot, t),\,\psi_T(s,\,\cdot)}_{\FH}\nonumber\\
   &=\frac{1}{T c_H^2}\int_{\Rnum}\mathcal{F}(e^{- \bar{\gamma}(\cdot-t)}\mathbf{1}_{\{0\le t\le \cdot\le T\}} )(\xi)\overline{\mathcal{F}(e^{- \bar{\gamma}(s-\cdot)}\mathbf{1}_{\{0\le \cdot\le s\le T\}} )(\xi)}\abs{\xi}^{1-2H}\dif \xi.\label{contract phi}
\end{align}
where ${\mathcal {F}}(\cdot)$ is the Fourier transform in $L^2(\Rnum^d)$, and 
$c_H=\big(\frac{2\pi}{\Gamma(2H+1)\sin(\pi H)} \big)^{\frac12}$. Denote by $H(\cdot)=\mathbf{1}_{(0,\infty)}(\cdot)$ the Heaviside function. Note that for any $t\in (0,T)$,
\begin{align*}
  \mathcal{F}(e^{- \bar{\gamma}(\cdot-t)}\mathbf{1}_{\{0\le t\le \cdot\le T\}} )(\xi)&=\int_{\Rnum}e^{-\mi \xi u}e^{-\bar{\gamma}(u-t)} \mathbf{1}_{\{0\le t\le u\le T\}}\dif u\\
  &= \int_{\Rnum}e^{-\mi \xi u}\mathbf{1}_{[0,T]}(u) \left[e^{-\bar{\gamma}(u-t)} H(u-t)\right]\,\dif u \\
  &= \int_0^{T}e^{-\mi \xi u}\left[\frac{1}{2\pi}\int_{\Rnum} e^{\mi (u-t)\eta}\frac{1}{\bar{\gamma}+\mi \eta}\,\dif \eta \right] \,\dif u.
\end{align*}
By Fubini's theorem, we have
\begin{align*}
  \mathcal{F}(e^{- \bar{\gamma}(\cdot-t)}\mathbf{1}_{\{0\le t\le \cdot\le T\}} )(\xi)
  &=\frac{1}{2\pi}\int_{\Rnum} e^{-\mi t\eta}\frac{1}{\bar{\gamma}+\mi \eta}\,\dif \eta \int_0^{T}e^{\mi \eta u -\mi \xi u} \, \dif u\\
  &=\frac{1}{2\pi}\int_{\Rnum} e^{-\mi t\eta}\left[\frac{1}{\bar{\gamma}+\mi \eta}\frac{e^{\mi (\eta-\xi)T}-1}{\mi (\eta-\xi)}\right]\,\dif \eta.
\end{align*}
Substituting the above result into equation~\eqref{contract phi}, we obtain that
\begin{align*}
  \phi(s,t)&=\frac{1}{4 T c_H^2\pi^2}\int_{\Rnum}\left(\int_{\Rnum} \frac{ e^{-\mi t\eta}}{\bar{\gamma}+\mi \eta}\frac{e^{\mi (\eta-\xi)T}-1}{\mi (\eta-\xi)} \,\dif \eta \right) \left(\int_{\Rnum} \frac{ e^{\mi t\eta'}}{\bar{\gamma}-\mi \eta'}\frac{e^{-\mi (\eta'-\xi)T}-1}{-\mi (\eta'-\xi)} \,\dif \eta' \right)\abs{\xi}^{1-2H}\dif \xi.
\end{align*}That is to say, $\phi(s,t)$ is the  inversion Fourier transformation of 
\begin{align*}
   h(\eta,\,\eta')=\frac{1}{ T c_H^2 }\int_{\Rnum}\left( \frac{1}{\bar{\gamma}+\mi \eta}\frac{e^{\mi (\eta-\xi)T}-1}{\mi (\eta-\xi)}  \right) \left( \frac{1}{\bar{\gamma}+\mi \eta'}\frac{e^{\mi (\eta'+\xi)T}-1}{\mi (\eta'+\xi)} \right) \abs{\xi}^{1-2H}\dif \xi.
\end{align*}
Thus, it follows from the inner product formula of \cite{pip} that
\begin{align}
  \norm{\phi(s,t) }^2_{\FH^2}&=\frac{1}{c_H^2}\int_{\Rnum^2}\abs{h(\eta,\,\eta')}^2\abs{\eta}^{1-2H}\abs{\eta'}^{1-2H}\,\dif \eta\dif \eta'\notag\\
  &\le \frac{C}{T^2}\int_{\Rnum^2}\frac{\abs{\eta}^{1-2H}}{\abs{\gamma}^2+\eta^2}\frac{\abs{\eta'}^{1-2H}}{\abs{\gamma}^2+\eta'^2} \left(\int_{\Rnum}  \frac{\abs{e^{\mi (\eta-\xi)T}-1}}{\abs{\eta-\xi}} \frac{\abs{e^{\mi (\eta'+\xi)T}-1}}{\abs{\eta'+\xi}} \abs{\xi}^{1-2H}\dif \xi \right)^2\,\dif \eta\dif \eta'\notag\\
  &=\frac{C}{T^2}\int_{\Rnum_{+}^2}\frac{\eta^{1-2H}}{\abs{\gamma}^2+\eta^2}\frac{\eta'^{1-2H}}{\abs{\gamma}^2+\eta'^2} \left(\int_{\Rnum_{+}}  \frac{\abs{e^{\mi (\eta-\xi)T}-1}}{\abs{\eta-\xi}} \frac{\abs{e^{\mi (-\eta'+\xi)T}-1}}{\abs{\xi-\eta'}} \abs{\xi}^{1-2H}\dif \xi \right)^2\,\dif \eta\dif \eta',\label{zhouhjbdsh}
\end{align} where the last line is from the symmetry. Since for any fixed $\alpha\in (0, \frac12)$ there exists a positive constant $C_{\alpha}$ such that for any $x\ge 0$,
\begin{align*}
    \abs{e^{\mi x}-1}\le C_{\alpha} x^{\alpha},
\end{align*}
we can rewrite the inequality \eqref{zhouhjbdsh} as follows:
\begin{align*}
    \norm{\phi(s,t) }^2_{\FH^2}&\le \frac{C}{T^{2-4\alpha}} \int_{\mathbb{R}_+^2, \eta \geq \eta'} g(\eta,\eta')\Big(\int_{\Rnum_{+}}f_{\alpha}(\xi,\eta,\eta') d\xi \Big)^2\,\dif \eta\dif \eta',
\end{align*} where \begin{equation}\label{falpha}
     f_{\alpha}(\xi,\eta,\eta') = |\xi-\eta|^{-1+\alpha} |\xi-\eta'|^{-1+\alpha} \xi^{1-2H}\mathbf{1}_{\{\eta' \le \eta \}},
   \end{equation} 
   and 
   \begin{equation}\label{geta}
     g( \eta,\eta') =  \frac{\eta^{1-2H}}{\abs{\gamma}^2+\eta^2} \frac{\eta'^{1-2H}}{\abs{\gamma}^2+\eta'^2},
   \end{equation}
respectively.  It follows from Lemma~\ref{coro-zhou} by choosing appropriate $\alpha$ values that 
\begin{align*}
  \norm{\phi(s,t) }^2_{\FH^{2}}&\le \frac{C}{T^{2-4\alpha}}, 
\end{align*}
when $H\in (\frac{1}{6},\,\frac12)$. Hence, $\phi(s,t)=\frac{1}{T}\psi_T{\otimes}_{1,0}h_T\to 0$ as $T\to \infty$.
\end{proof}

\noindent{\textbf{Proof of Theorem~\ref{com-ou parameter estimate}:\,}}\\
Recall equation (\ref{ratio.pro.re})
\begin{equation}\label{hat gamm-gamm}
    \hat{\gamma}_T-\gamma=-\frac{\frac{1}{ {T}}X_T}{\frac{1}{T}\int_0^T \abs{Z_t}^2\dif t}.
\end{equation}
Lemma~\ref{lem fenmuzixian} implies that $\frac{1}{T}\int_0^T \abs{Z_t}^2\dif t$ converges to $ H\Gamma(2H)d$ almost sure as $T\to \infty$,  where $d$ is given in Theorem~\ref{com-ou parameter estimate}.
Proposition~\ref{prop fenzi as converg} implies that  $\frac{1}{ {T}} X_T$ converges to zero almost sure as $T\to \infty$. Hence, $\hat{\gamma}_T$  converges to $\gamma$ almost surely
as $T\rightarrow \infty$.
Denote $F_T=\frac{1}{\sqrt{T}}X_T$.
From Theorem~\ref{com FMT}, Proposition~\ref{qq1a} and Proposition~\ref{con equa 0}, we have \[
\hbox{ $F_T$ converges in law to $\varpi\sim \mathcal{N}(0, 
   (H\Gamma(2H))^2\tensor{C})$,}
   \]
    where $\tensor{C}$ is given in Theorem~\ref{com-ou parameter estimate}. 
From equation (\ref{hat gamm-gamm}), we obtain
\begin{equation*}
   \sqrt{T}(\hat{\gamma}_T-\gamma)=-\frac{F_T}{\frac{1}{T}\int_0^T \abs{Z_t}^2\dif t}.
\end{equation*}
Therefore, it follows from Lemma~\ref{lem fenmuzixian} and Slutsky's theorem that
$ \sqrt{T}(\hat{\gamma}_T-\gamma)$ converges in distribution to the bivariate Gaussian vector $\mathcal{N}(0,\frac{1}{d^2}\tensor{C})$.
{\hfill\large{$\Box$}}

\section{The existence of the $\alpha$-fractional Brownian bridge with Hurst parameter $H\in (0,\frac12)$}\label{sec.5}

As we have seen from the proof of 
Thoerem~\ref{com-ou parameter estimate}, 
the inner product formula \eqref{innp fg3-00} has been applied intensively to validate the conditions in the fourth moment theorem. 
To further demonstrate the usefulness of this inner product formula, 
we will show another two applications in this section of computing the second moments for $\alpha$-order fBm 
and the $\alpha$-fractional bridges when $H\in (0,\frac12)$. 
The results have been summarized in Thoerem~\ref{fbm-bridge exist} 
and Theorem~\ref{fbm bridge 2 exist}. Again, we will start from providing several technical results.

\begin{lemma}\label{decomp.xi}
Assume $\alpha,\,H\in (0,1)$. Let the Gaussian process $(\xi_t)_{t\in [0,T)}$ be given in \eqref{xit dingyi}. For all $0\le s\le t< T$, we can decompose 
\begin{align}
\frac{1}{H}\E\big[(\xi_s-\xi_t)^2\big] &=J_1(s,t)+J_2(s,t)+J_3(s,t), \label{zuihougj bds}
\end{align}
where 
\begin{eqnarray*}
    J_1(s,t) &:=& 2H \int_{s}^{ t}(T-u)^{-\alpha-1}\dif u \int_{u}^t (T-v)^{-\alpha}  (v-u)^{2H-1} \dif v, \\
    J_2(s,t) &:=& (T-t)^{1-\alpha}\int_{s}^{ t}(T-u)^{-\alpha-1} (t-u)^{2H-1}\dif u, \\
    J_3(s,t) &:=& (T-s)^{-\alpha}  \int_{s}^t(T-v)^{-\alpha}  (v-s)^{2H-1}  \dif v.
\end{eqnarray*}
\end{lemma}
\begin{proof}
For simplicity, we assume that $H\in (0,\frac12)\cup (\frac12,1)$.
For all $0\le s\le t< T$, denote a function $f(u)=(T-u)^{-\alpha} \1_{[s,t]}(u).$
The inner product formula \eqref{innp fg3-00} implies
\begin{align}
\E\big[(\xi_s-\xi_t)^2\big]&=\E\left[ \int_s^t (T-u)^{-\alpha} \dif B^H_u\right]^2\notag \\
&=H\iint_{[0,T]^2}  f'(u) f(v)  \abs{v-u}^{2H-1}\sgn(v-u) \dif u\dif v.\label{chushids}
\end{align} 
where the distributional derivative $  f'(\cdot) $ is as follows:
\begin{align*}
f'(u)= \alpha (T-u)^{-\alpha-1} \cdot \1_{[s,t]}(u) +  (T-u)^{-\alpha} \cdot \big(\delta_s(u)-\delta_t(u)\big).
\end{align*} 
Then,
\begin{eqnarray}
\frac{1}{H}\E\big[(\xi_s-\xi_t)^2\big]
& = & \alpha\iint_{[s,t]^2}(T-u)^{-\alpha-1}(T-v)^{-\alpha}  \abs{v-u}^{2H-1}\sgn(v-u) \dif u\dif v \notag\\
& & + \ (T-t)^{-\alpha}  \int_{s}^t(T-v)^{-\alpha}  (t-v)^{2H-1}  \dif v \notag\\
& & + \ (T-s)^{-\alpha}  \int_{s}^t(T-v)^{-\alpha}  (v-s)^{2H-1}  \dif v \notag\\
&=: & \alpha I_1(s,t)+I_2(s,t)+J_3(s,t).\label{zhongjianzhy}
\end{eqnarray}
It is evident that
\begin{eqnarray*}
I_1(s,t ) 
&  = & \iint_{s\le u< v\le t}(T-u)^{-\alpha-1}(T-v)^{-\alpha}  (v-u)^{2H-1}  \dif u\dif v\\
& & -\iint_{s\le v< u\le t}(T-u)^{-\alpha-1}(T-v)^{-\alpha}  (u-v)^{2H-1}  \dif u\dif v\\
& = & \iint_{s\le u< v\le t}(T-u)^{-\alpha-1}(T-v)^{-\alpha}  (v-u)^{2H-1}  \dif u\dif v\\
& & -\iint_{s\le u< v\le t}(T-v)^{-\alpha-1}(T-u)^{-\alpha}  (v-u)^{2H-1}  \dif u\dif v\\
& = & -\iint_{s\le u< v\le t}(T-u)^{-\alpha-1}(T-v)^{-\alpha-1}  (v-u)^{2H}  \dif u\dif v.
\end{eqnarray*}
By Fubini's theorem and integration by parts, we have
\begin{align*}
\alpha I_1(s,t) 
&=-\int_{s}^{ t}(T-u)^{-\alpha-1}\dif u \int_{u}^t  (v-u)^{2H} \dif (T-v)^{-\alpha} \notag\\
&=\int_{s}^{ t}(T-u)^{-\alpha-1}\dif u \left[-(v-u)^{2H}(T-v)^{-\alpha}|_{v=u}^{v=t} + 2H\int_{u}^t (T-v)^{-\alpha} (v-u)^{2H-1} \dif v \right]\notag\\
& = 2H \int_{s}^{ t}(T-u)^{-\alpha-1}\dif u \int_{u}^t (T-v)^{-\alpha}  (v-u)^{2H-1} \dif v\notag \\
& \quad -(T-t)^{-\alpha}\int_{s}^{ t}(T-u)^{-\alpha-1} (t-u)^{2H}\dif u.
\end{align*}

This implies that 
\begin{align}
\alpha I_1(s,t)+I_2(s,t)&=2H \int_{s}^{ t}(T-u)^{-\alpha-1}\dif u \int_{u}^t (T-v)^{-\alpha}  (v-u)^{2H-1} \dif v\notag \\
& \quad +(T-t)^{1-\alpha}\int_{s}^{ t}(T-u)^{-\alpha-1} (t-u)^{2H-1}\dif u\notag\\
&:=J_1(s,t)+J_2(s,t).\label{zhongjzhy 2}
\end{align}Substituting equation \eqref{zhongjzhy 2} into equation \eqref{zhongjianzhy}, we obtain the desired result \eqref{zuihougj bds}. 
\end{proof}

The above decomposition result will be applied in the following proposition to study the increment of the Gaussian process defined in equation \eqref{xit dingyi}.

\begin{proposition}\label{dianzeduliang}
Assume $H,\,\alpha\in (0,1)$ and the Gaussian process $(\xi_t)_{t\in[0,T)}$ is given by \eqref{xit dingyi}. For any fixed $t\in (0,T)$, there exists a positive constant $C$ depending on $t,\,T$ such that
\begin{align}\label{gima2fang jie}
\sigma^2(u,v):=\E\big[(\xi_u-\xi_v)^2\big]\le C_{t,T}\abs{u-v}^{2H },\quad 0\le u, v\le t,
\end{align}where $\sigma^2(u,v)$ and $\sigma(u,v)$ are called the structure function and canonical metric for the process $\xi$.
Furthermore, if $\alpha\in (0,H)$, there exists a positive constant $C$ independent of $T$ such that
\begin{align}\label{gima2fang jie}
\sigma^2(s,t)=\E\big[(\xi_s-\xi_t)^2\big]\le C\abs{s-t}^{2(H-\alpha)},\quad 0\le s, t< T.
\end{align}
\end{proposition}
\begin{proof}
According to Lemma~\ref{decomp.xi}, for any $ 0\le u< v\le t$, $\E\big[(\xi_u-\xi_v)^2\big]$ can be decomposed into three terms. We will bound each of them as follows: 
\begin{align*}
J_1(u,v)&\le 2H (T-t)^{-2\alpha} \int_{u}^{ v}(T-x)^{-1} \dif x \int_{x}^v  (y-x)^{2H-1} \dif y< (T-t)^{-2\alpha} (v-u)^{2H}  ,\\ 
J_2(u,v)&=(T-v)^{1-\alpha} \int_{u}^{ v} (T-x)^{-\alpha-1} (v-x)^{2H-1}\dif x <  \frac{1}{2H}(T-t)^{-2\alpha} (v-u)^{2H} , \\
J_3(u,v)& =(T-u)^{-\alpha} \int_{u}^{ v} (T-x)^{-\alpha} (x-u)^{2H-1}  \dif x<  \frac{1}{2H}(T-t)^{-2\alpha} (v-u)^{2H}.
\end{align*} 
Thus,
\begin{equation*}
 \E\big[(\xi_u-\xi_v)^2\big] = \sum_{i=1,2,3} J_i(u, v)\le 2(T-t)^{-2\alpha} \abs{u-v}^{2H}.
\end{equation*}
Similarly,  when $\alpha\in (0,H)$ and $H\in (0,1)$, for any $0\le s<t<T$,  we can decompose $\E\big[(\xi_s-\xi_t)^2\big]$ according to Lemma~\ref{decomp.xi} and each component is bounded as follows:  
\begin{align*}
J_1(s,t)&\le 2H \int_{s}^{ t}(t-u)^{-\alpha-1}\dif u \int_{u}^t (t-v)^{-\alpha}  (v-u)^{2H-1} \dif v\\
&= 2H B(1-\alpha, 2H) \int_{s}^{ t} (t-u)^{2(H-\alpha)-1} \dif u,\\ 
J_2(s,t)&\le\int_{s}^{ t} (T-u)^{1-\alpha}(T-u)^{-\alpha-1} (t-u)^{2H-1}\dif u \le \int_{s}^{ t}   (t-u)^{2(H-\alpha)-1}\dif u ,\\
J_3(s,t)& \le\int_{s}^t (t-v)^{- \alpha}  (v-s)^{2H-\alpha-1}  \dif v=B(1-\alpha, 2H-\alpha) (t-s)^{2(H-\alpha)},
\end{align*} where $B(\cdot, \cdot)$ denotes the beta function. Then,
\begin{equation*}
 \E\big[(\xi_s-\xi_t)^2\big]\le H\left[ \frac{H B(1-\alpha, 2H) +\frac12}{H-\alpha} +B(1-\alpha, 2H-\alpha)  \right]\abs{t-s}^{2(H-\alpha)}.
\end{equation*}
\end{proof}
Next, we will apply Proposition~\ref{dianzeduliang} and the Kolmogorov-Centsov theorem to prove Theorem~\ref{fbm-bridge exist} and Theorem~\ref{fbm bridge 2 exist}.\\

{\bf Proof of Theorem~\ref{fbm-bridge exist}: }
It follows from equation \eqref{zuihougj bds} that 
\begin{align}
\E[\xi_T^2]=\lim_{t\uparrow T}\E[\xi_t^2] =H\times \left[J_1(0,T)+J_2(0,T)+J_3(0,T) \right],\label{qidiands}
\end{align}where $J_1,J_2, J_3$ are given in \eqref{zhongjianzhy} and \eqref{zhongjzhy 2}.
It is evident that
\begin{align}
J_3(0,T)&=T^{-\alpha}  \int_{0}^T(T-v)^{-\alpha}  v^{2H-1}  \dif v=B(1-\alpha, 2H)T^{2(H-\alpha)},\label{i3tjixian}\\
\frac{1}{2H}J_1(0,T)&=\int_{0}^{ T}(T-u)^{-\alpha-1}\dif u \int_{u}^T (T-v)^{-\alpha}  (v-u)^{2H-1} \dif v\notag \\
&=\frac{B(1-\alpha, 2H)}{2(H-\alpha)}T^{2(H-\alpha)}.\label{J1tjixian}
\end{align}
For the term $J_2(0, T)$, making the change of variables $y=T-t,\,x=t-u$ and applying Lebesgue's dominated theorem yield
\begin{align}
J_2(0,T)&=\lim_{t\uparrow T} (T-t)^{1-\alpha}\int_{0}^{ t}(T-u)^{-\alpha-1} (t-u)^{2H-1}\dif u\notag \\
&=\lim_{y\to 0+} y^{1-\alpha}\int_{0}^{ T}(y+x)^{-\alpha-1} x^{2H-1}\1_{[0,T-y]}(x)\dif x\notag \\
&=\lim_{y\to 0+}\int_{0}^{ T}\left(\frac{y}{y+x}\right)^{1-\alpha} \left(\frac{x}{y+x}\right)^{2\alpha}  x^{2(H-\alpha)-1} \dif x=0.\label{j0tjixian}
\end{align} 
 Plugging the results \eqref{i3tjixian}-\eqref{j0tjixian} into equation \eqref{qidiands}, we obtain the desired result \eqref{xiT2qiw}.
 {\hfill\large{$\Box$}} 


 {\bf Proof of Theorem~\ref{fbm bridge 2 exist}: }
It follows from equation \eqref{zuihougj bds} that 
\begin{align}
\E[\tilde{Y}_T^2]&=\lim_{t\uparrow T} (T-t)^{2(\alpha-H)} \E[\xi_t^2] \notag \\
&=\lim_{t\uparrow T} H\times (T-t)^{2(\alpha-H)} \left[J_1(0,t)+J_2(0,t)+J_3(0,t) \right].\label{qidiandszata}
\end{align}
For the term $J_3(0, t)$, we have
\begin{align}
\lim_{t\uparrow T}  (T-t)^{2(\alpha-H)}  J_3(0,t)&=\lim_{t\uparrow T}  (T-t)^{2(\alpha-H)}  T^{-\alpha}  \int_{0}^t(T-v)^{-\alpha}  v^{2H-1}  \dif v\notag \\
&=B(1-\alpha, 2H)   T^{2(H-\alpha)} \lim_{t\uparrow T}  (T-t)^{2(\alpha-H)}=0. \label{zuihou0001}
\end{align}
For the term $J_1(0, t)$, we have 
\begin{align*}
&\frac{1}{2H}\lim_{t\uparrow T}  (T-t)^{2(\alpha-H)}  J_1(0,t)\notag \\
&=\lim_{t\uparrow T}  (T-t)^{2(\alpha-H)}  \int_{0}^{ t}(T-u)^{-\alpha-1}\dif u \int_{u}^t (T-v)^{-\alpha}  (v-u)^{2H-1} \dif v.\notag\\
\end{align*} By the change of variables $x=\frac{T-v}{T-t},\,y=\frac{T-v}{T-u}$, we have
\begin{align}
&\frac{1}{2H}\lim_{t\uparrow T}  (T-t)^{2(\alpha-H)}  J_1(0,t)\notag \\
&= \int_{1}^{ \infty}x^{2(H-\alpha)-1}\dif x \int_0^1  (1-y)^{2H-1} y^{\alpha-2H}\dif y = \frac{1}{2(\alpha-H)}B(2H,\,1+\alpha-2H).
\end{align}
For the term $J_2(0, t)$, we have 
\begin{align*}
\lim_{t\uparrow T}  (T-t)^{2(\alpha-H)}  J_2(0, t)&=\lim_{t\uparrow T} (T-t)^{1+\alpha-2H}\int_{0}^{ t}(T-u)^{-\alpha-1} (t-u)^{2H-1}\dif u\notag
\end{align*}
Making the change of variables $z=\frac{T-t}{T-u}$ yields
\begin{align}
\lim_{t\uparrow T}  (T-t)^{2(\alpha-H)}  J_2(0, t)
&=\int_0^1  (1-z)^{2H-1} z^{\alpha-2H} \dif z\notag \\
&=B(2H,\,1+\alpha-2H) .\label{zata0tjixian}
\end{align}
Plugging these results \eqref{zuihou0001}-\eqref{zata0tjixian} into equation \eqref{qidiandszata}, we obtain the desired result \eqref{zetaT2qiw}.

Next, to obtain the identity \eqref{zetaT2qiw-2}, it suffices to show the following limit 
\begin{equation}
\lim_{t\to T} \E[B_s^H \tilde{Y}_t]=0
\end{equation}
holds for any fixed $s\in (0,T)$. 
Recall $\alpha \in (H, 1)$. First, we write
\begin{align*}
\lim_{t\to T} \E[B_s^H \tilde{Y}_t]&=\lim_{t\to T} (T-t)^{\alpha-H} \left[\E[B_s^H (\xi_t-\xi_s)] + \E[B_s^H  \xi_s]\right]\notag\\
 &=\lim_{t\to T} (T-t)^{\alpha-H} \E[B_s^H (\xi_t-\xi_s)]. \notag
\end{align*}
When $t>s$, $\xi_t-\xi_s=B^H(f)$ with $f(u)=(T-u)^{-\alpha} \1_{[s,t]}(u)$ and $B^H_s=B^H(g)$ with $g(u)=\1_{[0,s]}(u)$.
Since $\alpha\in (H, 1)$ and the intersection of the supports of two elements $f,\,g\in \FH$ is of Lebesgue measure zero, it follows from It\^{o}'s isometry and equation \eqref{innp fg3-zhicheng0} (see \cite{Mishura}) that 
for any fixed $s\in (0,T)$ and $t>s$, we have
\begin{align*}
\lim_{t\to T} \E[B_s^H \tilde{Y}_t]
 &= H(2H-1)\lim_{t\uparrow T}(T-t)^{\alpha-H} \int_s^t   (T-u)^{-\alpha} \dif u \int_0^s  (u-v)^{2H-2} \dif v \notag\\
 &= H \lim_{t\uparrow T}(T-t)^{\alpha-H} \int_s^t   (T-u)^{-\alpha} \left( u^{2H-1} -(u-s)^{2H-1}\right) \dif u.\end{align*}
 Clearly,
 \begin{align*}
 \lim_{t\to T}\abs{ \E[B_s^H \tilde{Y}_t]}&\le H \lim_{t\to T}(T-t)^{\alpha-H} \int_s^T   (T-u)^{-\alpha}(u-s)^{2H-1}  \dif u\notag\\
 &=H\times B(1-\alpha, 2H) (T-s)^{2H-\alpha} \lim_{t\to T}(T-t)^{\alpha-H} =0.
 \end{align*}
 {\hfill\large{$\Box$}}

\section{Appendix}\label{appendix}

 Lemma~\ref{upper bound F} is trivial and well known. Please refer to Lemma 3.3 of \cite{chenzhou2021}. 
\begin{lemma} \label{upper bound F}
 Assume $\beta>-1$.  There exists a constant $C>0$ such that for any  $s\in [0,\infty)$,
\begin{align*}
e^{-  s}\int_0^{s} e^{  r} r^{\beta }\dif r&\le C \times\big(s^{\beta+1}\mathbbm{1}_{[0,1]}(s) + s^{\beta}\mathbbm{1}_{ (1,\,\infty)}(s)\big).
\end{align*}
Especially, when $\beta\in (-1,0)$, 
\begin{align*}
e^{-  s}\int_0^{s} e^{  r} r^{\beta }\dif r\le C \times(1\wedge s^{\beta}).
\end{align*}
\end{lemma}

\begin{lemma}\label{asymptotic expansion key-0} 
    Suppose  $\beta\in(-1,0)$. The following asymptotic expansion
    \begin{align*}  
    e^{-T } \int_0^T e^x x^{\beta} \dif x =T^{\beta}-\beta T^{\beta-1} +O(T^{\beta-2}).
    \end{align*}
     holds as $T\to \infty$.
\end{lemma}
\begin{proof}
Let the function $g(x)=e^{x}x^{\beta -2}$. Since  $\lim_{x\to\infty}g(x)=\infty$ and $g'(x)\neq 0$ in the neighbourhood of $\infty$, 
then the lemma can be proved by applying L'H\^{o}pital's rule to show the limit
    \begin{align*}
        \lim_{T\to \infty} \frac{e^{-T } \int_0^T e^x x^{\beta} \dif x- (T^{\beta}-\beta T^{\beta-1})}{T^{\beta-2}}&=\lim_{T\to \infty} \frac{\int_0^T e^x x^{\beta} \dif x- e^{T}(T^{\beta}-\beta T^{\beta-1})}{e^{T}T^{\beta-2}}\\
        &= \beta(\beta-1).
    \end{align*}

\end{proof}
\begin{lemma}\label{asymptotic expansion key}
    Suppose  $\alpha_1>-1$ and $\delta = 1+\alpha_1+\alpha_2\in (-1,1)$. The following asymptotic expansion holds.
    \begin{align*}
  &  \int_{0\le x \le z \le T } e^{x-z} x^{\alpha_1}  z ^{\alpha_2}\dif x \dif z \\
  &  =\left\{
      \begin{array}{ll}
\Gamma(\delta+1)\mathrm{B}(1+\alpha_1, -\delta)+\delta^{-1}T^{\delta} +O(T^{\delta-1}), & \quad  \delta\in (-1, 0),\\
\log T +o(\log T), & \quad \delta= 0,\\
\delta^{-1}T^{\delta}+\alpha_2\Gamma(\delta) {B}(1+\alpha_1,1-\delta){-}\frac{\alpha_1}{\alpha_1+\alpha_2}T^{\delta-1}+O(T^{\delta-2}),& \quad \delta\in (0,1).
 \end{array}
\right.     
    \end{align*} 
\end{lemma}
\begin{proof}
When $\delta\in (-1,0)$, we must have $\alpha_2<0$. The change of variable $x=pz$ with $p\in (0,1)$ implies that
\begin{align*}
    \int_{0\le x \le z } e^{x-z} x^{\alpha_1}  z ^{\alpha_2}\dif x \dif z  &=\int_0^1 p^{\alpha_1}\dif p \int_0^{\infty} e^{-(1-p)z}z^{\delta}\dif z
    =\Gamma(1+\delta) {B}(1+\alpha_1, -\delta).
\end{align*}Hence, 
\begin{align*}
     \int_{0\le x \le z \le T } e^{x-z} x^{\alpha_1}  z ^{\alpha_2}\dif x \dif z-\Gamma(1+\delta) {B}(1+\alpha_1, -\delta)=-\int_{0\le x \le z, z>T } e^{x-z} x^{\alpha_1}  z ^{\alpha_2}\dif x \dif z
\end{align*}is $o(1)$.
Then the hypotheses of L'H\^{o}pital's rule are fulfilled and we have 
\begin{align*}
   \lim_{T\to\infty}  \frac{-\int_{0\le x \le z,z>T } e^{x-z} x^{\alpha_1}  z ^{\alpha_2}\dif x \dif z -\delta^{-1} T^{\delta}}{ T^{\delta-1}}&=  \lim_{T\to\infty}  \frac{ e^{-T}T^{\alpha_2}\int_{0}^T e^{x} x^{\alpha_1}  \dif x   -  T^{\delta-1}}{(\delta-1) T^{\delta-2}}\\
   &=\lim_{T\to\infty}  \frac{  \int_{0}^T e^{x} x^{\alpha_1}  \dif x   -  e^T T^{\alpha_1}}{(\delta-1) e^{T}T^{\alpha_1-1}}.
\end{align*}
The function $g(T)=e^{T}T^{\alpha_1-2}$ satisfies $g(\infty)=\infty$ and $g'(T)\neq 0$ in the neighbourhood of $\infty$. Applying L'H\^{o}pital's  rule to the above ratio again yields
 \begin{align*}
    \lim_{T\to\infty}  \frac{  \int_{0}^T e^{x} x^{\alpha_1}  \dif x   -  e^T T^{\alpha_1}}{(\delta-1) e^{T}T^{\alpha_1-1}}=-\frac{\alpha_1}{\delta-1}.
 \end{align*}
This finishes the proof for the case $\delta\in (-1,0)$.

When $\delta=0$, $-1-\alpha_2=\alpha_1$. It is clear that  the functions $g(T)=\log T$ and $h(T)=e^T T^{\alpha_1}$ fulfill the hypotheses of L'H\^{o}pital's rule. Then we have
 \begin{align*}
    \lim_{T\to \infty} \frac{\int_{0\le x \le z \le T } e^{x-z} x^{\alpha_1}  z ^{\alpha_2}\dif x \dif z -\log T}{\log T}&=\lim_{T\to \infty} \frac{ e^{-T}T^{\alpha_2}\int_0^T e^xx^{\alpha_1}  \dif x -T^{-1}}{T^{-1}}\\
    &=\lim_{T\to \infty} \frac{ \int_0^T e^xx^{\alpha_1}  \dif x -e^T T^{\alpha_1}}{e^T T^{\alpha_1}}=0.
 \end{align*} Hence, the lemma is proved for the case $\delta=0$.

When $\delta\in (0,1)$, the result can be obtained using integration by parts and Lemma~\ref{asymptotic expansion key-0}: 
\begin{align*}
   & \int_{0\le x \le z \le T } e^{x-z} x^{\alpha_1}  z ^{\alpha_2}\dif x \dif z\\
    &=\delta^{-1}T^{\delta}-T^{\alpha_2}e^{-T}\int_0^T e^{x} x^{\alpha_1}\dif x+\alpha_2\int_{0\le x \le z \le T } e^{x-z} x^{\alpha_1}  z ^{\alpha_2-1}\dif x \dif z\\
    &=\delta^{-1}T^{\delta}-T^{\alpha_2}\big(T^{\alpha_1}-\alpha_1T^{\alpha_1-1}\big)+\alpha_2 \left[\Gamma(\delta) {B}(1+\alpha_1, 1-\delta)+ \frac{1}{\delta-1}T^{\delta-1} +O(T^{\delta-2})\right]\\ &=\delta^{-1}T^{\delta}+\alpha_2\Gamma(\delta) {B}(1+\alpha_1,1-\delta){-}\frac{\alpha_1}{\alpha_1+\alpha_2}T^{\delta-1}+O(T^{\delta-2}).
\end{align*} 
\end{proof}
\begin{lemma}\label{coro key point000}
Let $\kappa$ be as in \eqref{kapp notation}.   When $H\in (\frac{1}{4},\frac12)$ and $\gamma=\lambda- \mi \omega$ with $\lambda>0,\,\omega\in \Rnum$, we have  
    \begin{align}
 \int_{0\le x \le z \le T } e^{\gamma(x-z)} x^{2H-1}  z ^{2H}\dif x \dif z&=  
 \frac{1}{4H\gamma }  T^{4H}+\frac{1-2H}{(4H-1)\gamma^2} T^{4H-1}+ \frac{2H\kappa}{ \gamma^{1+4H}}  \notag\\
    & +\frac{H-1}{\gamma^3}T^{4H-2}+O(T^{4H-3}), \label{changyongjifen} \\
  \int_{0\le x \le z \le T } e^{\gamma(x-z)} x^{2H}  z ^{2H-1}\dif x \dif z&=  
 \frac{1}{4H\gamma }  T^{4H}-\frac{2H}{(4H-1)\gamma^2} T^{4H-1}- \frac{2H\kappa}{ \gamma^{1+4H}}\notag\\
    & +\frac{H}{\gamma^3}T^{4H-2}  +O(T^{4H-3}), \label{changyongjifen-000} \\
  \int_{0\le x \le z \le T } e^{\gamma(x-z)} x^{2H-1}  z ^{2H}(z-x)\dif x \dif z &=\frac{1}{4H\gamma^2 }  T^{4H}+\frac{2(1-2H)}{(4H-1)\gamma^3} T^{4H-1}\notag \\
  &+\frac{2H (4H+1)\kappa}{ \gamma^{2+4H}}+\frac{H}{\gamma^3}T^{4H-2} +O(1).\label{changyongjifen-buch-000}
    \end{align}
\end{lemma}
\begin{proof} 
It suffices to show that equation \eqref{changyongjifen} holds for $\omega=0$, i.e, $\gamma>0$.  Using integration by parts and making the change of variable $u=\gamma x, v=\gamma z$ yield
  \begin{align*}
    & \int_{0\le x \le z \le T } e^{\gamma(x-z)} x^{2H-1}  z ^{2H}\dif x \dif z=-\frac{1}{\gamma^{4H+1}}\int_0^{\gamma T} e^u u^{2H-1}\dif u\int_{u}^{\gamma T} v^{2H}\dif e^{-v}\\
    &=\frac{1}{\gamma^{4H+1}}\left[-(\gamma T)^{2H}e^{-\gamma T}\int_0^{\gamma T} e^u u^{2H-1}\dif u +\int_0^{\gamma T} u^{4H-1}\dif u\right.\\
    &\left.+2H \int_{0\le u \le v \le  \gamma T } e^{ u-v} u^{2H-1}  v ^{2H-1}\dif u \dif v\right]\\
    &={\frac{1}{4H\gamma }  T^{4H}+\frac{1-2H}{(4H-1)\gamma^2} T^{4H-1} +2H\kappa \gamma^{-1-4H}}  +\frac{H-1}{\gamma^3}T^{4H-2}+O(T^{4H-3}),
\end{align*}  where in the last line we use Lemma~\ref{asymptotic expansion key-0} and Lemma~\ref{asymptotic expansion key}. Equation \eqref{changyongjifen-000} can be shown similarly.

Next, we will prove \eqref{changyongjifen-buch-000}. Using integration by parts yields
  \begin{align*}
    & \int_{0\le x \le z \le T } e^{\gamma(x-z)} x^{2H-1}  z ^{2H+1}\dif x \dif z=-\frac{1}{\gamma}\int_0^{  T} e^{\gamma x} x^{2H-1}\dif x\int_{x}^{ T} z^{2H+1}\dif e^{-\gamma z}\\
    &=\frac{1}{\gamma}\left[-T^{2H+1}e^{-\gamma T}\int_0^{T} e^{\gamma x} x^{2H-1}\dif x +\int_0^{  T} x^{4H}\dif x \right.\\
    & \left. +(2H+1) \int_{0\le x \le y \le    T } e^{ \gamma (x-z)} x^{2H-1}  z ^{2H-1}\dif x \dif z\right],
  \end{align*}
  and
  \begin{align*}
    & \int_{0\le x \le z \le T } e^{\gamma(x-z)} x^{2H}  z ^{2H}\dif x \dif z=\frac{1}{\gamma}\int_0^{  T} e^{-\gamma z} z^{2H}\dif z\int_{0}^{ z} x^{2H}\dif e^{\gamma x} \\
    &=\frac{1}{\gamma}\left[\int_0^{  T} x^{4H}\dif x-2H \int_{0\le x \le y \le T} e^{ \gamma (x-z)} x^{2H-1}  z ^{2H-1}\dif x \dif z\right].
\end{align*} Hence, Lemma~\ref{asymptotic expansion key-0} and equation \eqref{changyongjifen} imply
\begin{align*}
 & \int_{0\le x \le z \le T } e^{\gamma(x-z)} x^{2H-1}  z ^{2H}(z-x)\dif x \dif z \\
 &=\frac{1}{\gamma}\left( (4H+1) \int_{0\le x \le y \le    T } e^{ \gamma (x-z)} x^{2H-1}  z ^{2H-1}\dif x \dif z -T^{2H+1}e^{-\gamma T}\int_0^{T} e^{\gamma x} x^{2H-1}\dif x\right) \\
 &=\frac{1}{4H\gamma^2 }  T^{4H}+\frac{2(1-2H)}{(4H-1)\gamma^3} T^{4H-1}+\frac{2H (4H+1)\kappa}{ \gamma^{2+4H}}+\frac{H}{\gamma^3}T^{4H-2} +O(1).
\end{align*}
\end{proof}

\begin{lemma}\label{est.f}
   Suppose that $f_{\alpha}(\xi,\eta,\eta')$ is given by equation \eqref{falpha}.
   We have the following results:
  \begin{description}
   \item [(i)] If $\alpha \in (0,\frac12)$, there exists some positive constant $K_1:=K_1(\alpha, H)$ such that
  \begin{equation*}
	\int_{(0,\eta+\eta')}f_{\alpha}(\xi,\eta,\eta') d\xi \leq K_1 \eta^{1-2H} (\eta-\eta')^{-1+2\alpha} \,.
  \end{equation*}

  \item [(ii)] If $\alpha \in (0,H)$, there exists some positive constant $K_2:=K_2(\alpha, H)$ such that
  \begin{eqnarray*}
	\int_{[\eta+\eta', \infty)} f_{\alpha}(\xi,\eta,\eta') d\xi \leq K_2 (\eta')^{2\alpha-2H} \,.
  \end{eqnarray*}
  \end{description}
\end{lemma}

\begin{proof}
   (i): We  partition $(0,\eta+\eta')$ into three intervals: $(0, \eta'] \cup (\eta', \eta] \cup (\eta, \eta+\eta')$.  For $\xi \in (0, \eta']$, we make a change of variable $\eta'-\xi \to (\eta - \eta')x$. Since $\alpha\in (0,\frac12)$, we can obtain
\begin{eqnarray}
  \int_0^{\eta'} f_{\alpha}(\xi,\eta,\eta') d \xi & \leq & (\eta')^{1-2H} (\eta - \eta')^{-1+2\alpha} \int_0^{\frac{\eta'}{\eta - \eta'}} x^{-1+\alpha} (1+x)^{-1+\alpha} dx \nonumber\\
  & \leq & C \eta^{1-2H} (\eta - \eta')^{-1+2\alpha} \label{fxi-4}.
\end{eqnarray}
For $\xi \in (\eta', \eta]$, observe that
\begin{eqnarray}
  \int_{\eta'}^{\eta} f_{\alpha}(\xi,\eta,\eta') d \xi & \leq & \eta^{1-2H} \int_{\eta'}^{\eta} (\eta-\xi)^{-1+\alpha} (\xi-\eta')^{-1+\alpha} d\xi \nonumber \\
  & = &  {\rm B}(\alpha, \alpha) \eta^{1-2H} (\eta - \eta')^{-1+2\alpha}. \label{fxi-3}
\end{eqnarray}
For $\xi \in (\eta, \eta' + \eta]$, we make a change of variable $\xi - \eta \to (\eta - \eta')x$ to obtain
 \begin{eqnarray}
  \int_{\eta}^{\eta+\eta'} f_{\alpha}(\xi,\eta,\eta') d\xi & \leq & (2\eta)^{1-2H} (\eta - \eta')^{-1+2\alpha} \int_0^{\frac{\eta'}{\eta - \eta'}} x^{-1+\alpha} (1+x)^{-1+\alpha}  dx \nonumber \\
  &\leq& C \eta^{1-2H} (\eta - \eta')^{-1+2\alpha}, \label{fxi-2}
 \end{eqnarray}where the last inequality is from $\alpha\in (0,\frac12)$.

By  \eqref{fxi-4}, \eqref{fxi-3}, and \eqref{fxi-2}, the first part of lemma is obtained.

(ii) If $\alpha \in (0,H)$, then
  \begin{eqnarray}
   \int_{\eta+\eta'}^{\infty} f_{\alpha}(\xi,\eta,\eta') d\xi & = & \int_0^{\infty} (x+\eta')^{-1+\alpha} (x+\eta)^{-1+\alpha} (x+\eta+\eta')^{1-2H} dx \nonumber \\
   &\leq& 2^{1-2H} \int_0^{\infty} (x+\eta')^{-1+\alpha} (x+\eta)^{-1+\alpha} (x+\eta)^{1-2H} dx \nonumber \\
   &\leq &  \int_0^{\infty} (x+\eta')^{-1+2\alpha-2H} dx =  K(\eta')^{2\alpha-2H},
   \end{eqnarray}
and this finishes the proof of the second part of lemma.
\end{proof}  		
\begin{lemma}\label{coro-zhou}
     Suppose that $f_{\alpha}(\xi,\eta,\eta')$ is defined by equation \eqref{falpha} and $g( \eta,\eta')$ is defined by equation \eqref{geta}.
   For $H \in (0, \frac{1}{2})$, we have the following results.

  \begin{description}
   \item [(i)] If $H\in (\frac{1}{6},\,\frac12)$ and $\alpha_1 \in (\frac14,\frac12\wedge \frac{3H}{2})$, then the following integral is finite:
  \begin{equation*}
	\int_{\mathbb{R}_+^2, \eta \geq \eta'} g(\eta,\eta')\Big(\int_{(0,\eta+\eta')}f_{\alpha_1}(\xi,\eta,\eta') d\xi \Big)^2<\infty.
  \end{equation*}

  \item [(ii)] If $\alpha_2 \in \big(0\vee (\frac32 H-\frac12)\,,H\big)$, then the following integral is finite:
  \begin{equation*}
	\int_{\mathbb{R}_+^2, \eta \geq \eta'} g(\eta,\eta')\Big(\int_{[\eta+\eta',\,\infty)}f_{\alpha_2}(\xi,\eta,\eta') d\xi \Big)^2<\infty.
  \end{equation*}
  \end{description}
\end{lemma}
\begin{proof}
(i): Choose a real number $q$ such that
\begin{equation*}
   2\alpha_1+1-4H< q<1-H< 2\alpha_1+2-4H.
\end{equation*}
Then it follows from (i) of Lemma~\ref{est.f} that
\begin{eqnarray*}
	&\quad& \int_{\mathbb{R}_+^2, \eta \geq \eta'} g(\eta,\eta')\Big(\int_{(0,\eta+\eta')}f_{\alpha_1}(\xi,\eta,\eta') d\xi \Big)^2\\
 &\le &C \int_{\mathbb{R}_+^2, \eta \geq \eta'} g(\eta,\eta')\left[\eta^{1-2H} (\eta-\eta')^{-1+2\alpha}\right]^2   \\
 &\quad & (\text{let} \quad \eta'= x \eta)\\
 &=&C \int_0^\infty \dif \eta \frac{\eta^{1-2H}}{1+\eta^2}\int_0^1 \frac{(\eta x)^{1-2H}}{1+\eta^2 x^2} \eta^{-2+4\alpha_1}(1-x)^{-2+4\alpha_1} \eta^{2-4H}\eta \dif x \\
	& \leq & C \int_0^\infty \frac{\eta^{3-8H+4\alpha_1}}{1+\eta^2} \dif \eta \int_0^1 \frac{x^{1-2H}(1-x)^{-2+4\alpha_1}}{(\eta^2x^2)^{q}} \dif x < \infty \,.
\end{eqnarray*}
   (ii): It is a direct consequence of Lemma~\ref{est.f} (ii).
\end{proof}

\begin{lemma}\label{last key 00}
Let $\gamma,\,\lambda,\,\kappa $ be as above and $q(t,s):= \mathbf{1}_{\{0\le s\le  t\le T\}} $. When $H\in (\frac14,\frac12)$,
    \begin{align}
    &  \int_{[0,T]^4,\, t_2\le t_1}  e^{-\bar{\gamma}(t_1-s_1)-\gamma(t_2-s_2)}q(t_1,s_1)q(t_2,s_2) (t_1-t_2)^{\beta}  \sgn(s_1-s_2)  \abs{s_1-s_2}^{\beta } \dif \vec{t}  \dif \vec{s}  \notag\\
  &=\frac{T^{4H}}{4H(4H-1)\abs{\gamma}^2}-\frac{ T}{\lambda }\left[\frac{\Gamma^2(2H)}{2{\bar{\gamma}^{4H}}}  - \RE\frac{\kappa}{\gamma^{4H}} \right]-\frac{\lambda }{(4H-1) \abs{\gamma}^4} T^{4H-1}+O(1),\label{last jielun} 
\end{align}
and 
\begin{align}
&\int_{[0,T]^4, t_1\le s_2}  e^{-\bar{\gamma}(t_1-s_1)-\bar{\gamma}(s_2-t_2)}q(t_1,s_1)q(s_2,t_2) (s_2-s_1)^{\beta }\sgn(t_1-t_2)  \abs{t_2-t_1}^{\beta}  \dif \vec{t}  \dif \vec{s}\notag\\
&=\frac{-T^{4H}}{4H(4H-1)\bar{\gamma}^2} +\frac{2H(\Gamma^2(2H)-2\kappa)}{\bar{\gamma}^{4H+1}}T+ \frac{1}{(4H-1)\bar{\gamma}^3}T^{4H-1} +O(1).  \label{last jielun st}  
\end{align}
\end{lemma}
\begin{proof}
    We denote the first integral as $R(T)$ 
    and decompose it into the sum of integrals $R_i(T)$ over the disjoint regions $\Delta_i,\,i=1,2,3$:
    \begin{align}\label{rt fenjie}
        R(T) =R_{1}(T)+  R_2(T)+R_3(T),
    \end{align}where 
    \begin{align*}
\Delta_1&=\set{0\le  s_1\le s_2\le t_2\le t_1 \le T },\\
\Delta_2&=\set{0\le s_2\le t_2 \le s_1\le t_1 \le T }, \\
\Delta_3&=\set{0\le s_2\le s_1\le t_2\le t_1 \le T } .
\end{align*}
 Making change of variables $x=\abs{s_1-s_2},\, y=t_1-t_2,\,z=\abs{t_2-(s_1\vee s_2)}$, we obtain
 \begin{align*}
    {R_{1}(T)}&=-{\int_{0\le x+y+z\le T}e^{-\bar{\gamma}(x+  y)-2\lambda z}x^{\beta}y^{\beta}}\big(T-(x+y+z)\big)\dif x\dif y\dif z,\\
    {R_2(T)}&= \int_{0\le z\le x\wedge y\le x+y-z\le T}e^{-\bar{\gamma}y-\gamma x+2\lambda z} x^{\beta}y^{\beta}\big(T-(x+y-z)\big)\dif x\dif y\dif z,\\
    {R_3(T)}&={\int_{0\le x+y+z\le T}e^{-\bar{\gamma}x-\gamma y-2\lambda z}x^{\beta}y^{\beta}}\big(T-(x+y+z)\big) \dif x\dif y\dif z.
 \end{align*} For the term ${R_2(T)}$,  we have
\begin{align*}
      {R_2(T)}&=\int_{[0,T]^2}e^{-\bar{\gamma}y-\gamma x} x^{\beta}y^{\beta}\dif x\dif z \int_{0\vee(x+y-T)}^{x\wedge y} e^{2\lambda z} \big(T-(x+y-z)\big) \dif z.\end{align*}
Applying integration by parts yields
\begin{align*}
    \int_{0\vee(x+y-T)}^{x\wedge y} e^{2\lambda z} \big(T-(x+y-z)\big) \dif z&=\frac{1}{2\lambda}\Bigg[
    -e^{0\vee \big(2\lambda (x+y-T)\big)}\big(T-x-y+ 0\vee (x+y-T)\big)\notag\\
    &+e^{2\lambda (x\wedge y)}(T-x\vee y) -\frac{1}{2\lambda}\Big[e^{2\lambda (x\wedge y)}-e^{0\vee \big(2\lambda (x+y-T)\big)}\Big]\Bigg].\end{align*}  Then we have 
      \begin{align*}
   R_2(T)&=\frac{1}{2\lambda}\Bigg[\int_{[0,T]^2}e^{-\bar{\gamma}y-\gamma x+2\lambda (x\wedge y))} x^{\beta}y^{\beta}\left(T-x\vee y -\frac{1}{2\lambda}\right)\dif x\dif y\notag\\
   &-\int_{0\le x+y\le T}e^{-\bar{\gamma}y-\gamma x } x^{\beta}y^{\beta}(T-x-y)\dif x\dif y\Bigg]+O(1). 
\end{align*}
By the symmetry and the inequality  
\begin{equation*}
   \abs{\int_{[0,T]^2,x+y> T}e^{-\bar{\gamma}y-\gamma x } x^{\beta}y^{\beta}(T-x-y)\dif x\dif y }\le \frac{1}{4H}e^{-\lambda T}(2T)^{4H+1}=o(1),
\end{equation*}
 we have 
\begin{align}\label{r2t zhoujianjg}
  R_2(T) &=\frac{1}{\lambda}\Bigg[\RE \int_{0\le x\le y\le T}e^{\bar{\gamma}(x-y) } x^{\beta}y^{\beta}(T-y -\frac{1}{2\lambda})\dif x\dif y\notag\\
   &-\frac12 T \int_{[0,T]^2}e^{-\bar{\gamma}y-\gamma x } x^{\beta}y^{\beta}\dif x\dif y \Bigg]+O(1).\end{align}
Lemma~\ref{asymptotic expansion key} implies that
    \begin{align}\label{zhjjg 0000}
   \int_{0\le x\le y\le T}e^{\bar{\gamma}(x-y) } x^{\beta}y^{1+\beta} \dif x\dif y &=  \frac{T^{4H}}{4H\bar{\gamma}}- \frac{T^{4H-1}}{\bar{\gamma}^2}+\frac{2H}{\bar{\gamma}}\int_{0\le x\le y\le T}e^{\bar{\gamma}(x-y) } x^{\beta}y^{\beta} \dif x\dif y +O(1).\end{align}
Plugging \eqref{zhjjg 0000} into \eqref{r2t zhoujianjg} yields
\begin{align*}
  R_2(T) &=\frac{1}{\lambda}\left[\RE \left[ -\frac{T^{4H}}{4H\bar{\gamma}}+ \frac{T^{4H-1}}{\bar{\gamma}^2}+\left(T-\frac{1}{2\lambda}-\frac{2H}{\bar{\gamma}}\right) \int_{0\le x\le y\le T}e^{\bar{\gamma}(x-y) } x^{\beta}y^{\beta} \dif x\dif y\right]\right.\notag\\
   &\left.-\frac12 T \int_{[0,T]^2}e^{-\bar{\gamma}y-\gamma x } x^{\beta}y^{\beta}\dif x\dif y \right]+O(1).\end{align*}  Applying Lemma~\ref{asymptotic expansion key-0} yields 
   \begin{align}
   R_2(T) &=\frac{1}{\lambda}\RE\Bigg[\left(T-\frac{1}{2\lambda}-\frac{2H}{\bar{\gamma}}\right)\left[ \frac{1}{(4H-1)\bar{\gamma}}T^{4H-1}+\frac{\kappa}{\bar{\gamma}^{4H}}-\frac{1}{2\bar{\gamma}^{2}}T^{4H-2} \right]\notag\\
   &- \frac{T^{4H}}{4H\bar{\gamma}}+ \frac{T^{4H-1}}{\bar{\gamma}^2} -\frac{T}{2 }\Gamma^2(2H) \frac{1}{\abs{\gamma}^{4H}}\Bigg]+O(1)\notag\\
   &=\frac{T^{4H}}{4H(4H-1)\abs{\gamma}^2}+\frac{ T}{\lambda }\left[\kappa \RE\gamma^{-4H} -\frac12\Gamma^2(2H)  {\abs{\gamma}^{-4H}}\right]\notag\\
   &-\frac{T^{4H-1}}{2(4H-1)\lambda}\RE\left[\frac{1}{\gamma^2}+\frac{1}{ \abs{\gamma}^2} \right]+O(1).\label{last yansuan 1}
\end{align}

Using the following integration result 
\begin{align*}
    \int_0^{T-x-y}e^{-2\lambda z} \big(T-(x+y+z)\big) \dif z=\frac{1}{2\lambda}\Big[T-x-y-\frac{1}{2\lambda}\big(1-e^{-2\lambda (T-x-y)}\big)\Big],
\end{align*}
we obtain
\begin{align}
    {R_{1}(T)}+{R_3(T)}&=\frac{T}{2\lambda}\left[{\int_{0\le x+y\le T}e^{-\bar{\gamma}x-\gamma y }x^{\beta}y^{\beta}} \dif x\dif y -{\int_{0\le x+y \le T}e^{-\bar{\gamma}(x+  y) }x^{\beta}y^{\beta}} \dif x\dif y\right]+O(1)\notag\\
    &=\frac{T}{2\lambda}\left[{\int_{[0, T]^2}e^{-\bar{\gamma}x-\gamma y }x^{\beta}y^{\beta}} \dif x\dif y -{\int_{[0, T]^2}e^{-\bar{\gamma}(x+ y )}x^{\beta}y^{\beta}} \dif x\dif y\right]+O(1)\notag\\
    &=\frac{T}{2\lambda}\Gamma^2(2H)\left(\frac{1}{\abs{\gamma}^{4H}}- \frac{1}{\bar{\gamma}^{4H}}\right)+O(1),\label{last yansuan 2}
\end{align}
where in the second step we have applied the inequality
\begin{equation*}
   \abs{\int_{[0,T]^2,x+y> T}\big(e^{-\bar{\gamma}x-\gamma y } - e^{-\bar{\gamma}  (x+y) }\big) x^{\beta}y^{\beta}\dif x\dif y }\le \frac{1}{4H}e^{-\lambda T}(2T)^{4H}=o(1).
\end{equation*}
Plugging equation \eqref{last yansuan 1} and equation \eqref{last yansuan 2} into equation \eqref{rt fenjie} yields the desired result \eqref{last jielun}.

We denote the second integral as $S(T)$ 
    and decompose it into the sum of integrals $S_i(T)$ over the disjoint regions $\Delta_i',\,i=1,2,3$:
    \begin{align}\label{st fenjie 00 1}
        S(T) =S_{1}(T)+  S_2(T)+S_3(T),
    \end{align}where 
    \begin{align*}
\Delta_1'&=\set{0\le s_1\le t_2 \le t_1 \le s_2\le T },\\
\Delta_2'&= \set{0\le t_2\le  s_1\le t_1 \le s_2 \le T },\\
\Delta_3'&=\set{0\le s_1\le t_1\le t_2 \le s_2 \le T } .
\end{align*}
Making change of variables $x=\abs{t_1-t_2}, y=s_2-s_1$ and $z=(t_1\vee t_2)-s_1$ implies
\begin{align*}
   S_{1}(T)
   &=\int_{0\le x \le y\le T} e^{-\bar{\gamma} (x+y)}x^{\beta}y^{\beta}(T-y)(y-x)\dif x\dif y\\
   &=T\int_{0\le x \le y\le T} e^{-\bar{\gamma} (x+y)}x^{\beta}y^{\beta} (y-x)\dif x\dif y+O(1)\\
   &=\frac12 T\int_{[0,T]^2} e^{-\bar{\gamma} (x+y)}x^{\beta}y^{\beta} \abs{y-x}\dif x\dif y+O(1),\\
   S_{2}(T)&= \int_{0\le z\le x\wedge y\le x+y-z\le T}e^{-\bar{\gamma} (x+y)}x^{\beta}y^{\beta}\big(T-(x+y-z)\big)\dif x\dif y\dif z,\\
   S_{3}(T)
   &=-\int_{0\le x \le y\le T}e^{\bar{\gamma} (x-y)}x^{\beta}y^{\beta}(T-y)(y-x)\dif x\dif y.
\end{align*}
For the term $ S_{2}(T)$, it is clear the integral is $O(1)$ when $x+y>T$ and hence 
\begin{align*}
     S_{2}(T)&=\int_{[0,T]^2,x+y\le T}e^{-\bar{\gamma} (x+y)}x^{\beta}y^{\beta} \Big[\frac12 (T-x-y+z)^2|_{z=0}^{z=x\wedge y}\Big]\dif x\dif y+O(1)\\
     &=\int_{[0,T]^2,x+y\le T}e^{-\bar{\gamma} (x+y)}x^{\beta}y^{\beta} \Big[\frac12  (x\wedge y)(2T-x\vee y-x-y)\Big]\dif x\dif y+O(1)\\
     &=T\int_{[0,T]^2,x+y\le T}e^{-\bar{\gamma} (x+y)}x^{\beta}y^{\beta}    (x\wedge y) \dif x\dif y+O(1)\\
     &=T\int_{[0,T]^2}e^{-\bar{\gamma} (x+y)}x^{\beta}y^{\beta}    (x\wedge y) \dif x\dif y+O(1).
\end{align*}
By making change of variable $u=x+y$, we have 
\begin{align}   S_1(T)+S_2(T)&=T \int_{[0,T]^2}e^{-\bar{\gamma} (x+y)}x^{\beta}y^{\beta}    \big(\frac12\abs{y-x}+x\wedge y\big) \dif x\dif y +O(1)\notag\\
&=\frac12 T \int_{[0,T]^2,x+y\le T}e^{-\bar{\gamma} (x+y)}x^{\beta}y^{\beta}     (x+y)  \dif x\dif y +O(1)\notag\\
&=\frac{T}{2} B(1+\beta,1+\beta)\int_0^Te^{-\bar{\gamma} u} u^{2\beta+2}\dif u+O(1)\notag\\
&=\frac{2H\Gamma^2(2H)}{\bar{\gamma}^{1+4H}}T +O(1).\label{s12 t jifen}
\end{align}
Applying Lemma~\ref{coro key point000} yields when $H\in (\frac{1}{4}, \frac12)$
\begin{align}
    S_3(T) &=-T \int_{0\le x \le y\le T}e^{\bar{\gamma} (x-y)}x^{\beta}y^{\beta}(y-x)\dif x\dif y\notag\\
    &+\int_{0\le x \le y\le T}e^{\bar{\gamma} (x-y)}x^{\beta}y^{\beta+1}(y-x)\dif x\dif y\notag\\
    &=-T\Big[\frac{1}{(4H-1)\bar{\gamma}^2} T^{4H-1}+ \frac{4H\kappa}{ \bar{\gamma}^{1+4H}} -\frac{1}{\bar{\gamma}^3}T^{4H-2}+O(T^{4H-3})\Big]\notag\\
    &+\Big[\frac{1}{4H\bar{\gamma}^2 }  T^{4H}+\frac{2(1-2H)}{(4H-1)\bar{\gamma}^3} T^{4H-1} +O(1) \Big]\notag\\
    &=-\frac{T^{4H}}{4H(4H-1)\bar{\gamma}^2} -\frac{4H\kappa}{\bar{\gamma}^{4H+1}}T+ \frac{1}{(4H-1)\bar{\gamma}^3}T^{4H-1} +O(1).\label{s3t jifen}
\end{align}
Plugging equation \eqref{s12 t jifen} and equation \eqref{s3t jifen} into equation \eqref{st fenjie 00 1} yields the desired result \eqref{last jielun st}.

\end{proof}

\section*{Acknowledgement}
We would like to thank two anonymous referees for their helpful comments that allowed us to improve the presentation of the results.




\end{document}